\documentclass[12pt, a4paper]{amsart}

\usepackage{fullpage, amsmath, amsfonts, amsthm, amssymb, stmaryrd, color}
\usepackage{hyperref}
\hypersetup{
    urlcolor=blue,
    colorlinks=true,
    linkcolor= blue,
    citecolor=blue,
}

\theoremstyle{plain}
\newtheorem{Theorem}{Theorem}[section]
\newtheorem{Lemma}[Theorem]{Lemma}
\newtheorem{Corollary}[Theorem]{Corollary}
\newtheorem{Proposition}[Theorem]{Proposition}

\theoremstyle{remark}
\newtheorem{Remark}[Theorem]{Remark}

\newtheorem{Example}[Theorem]{Example}

\usepackage[utf8]{inputenc}

\newcommand{\hhb}[1]{{\hbox to#1pt{}}}

\newcommand{\mfp}{\mathfrak{p}}

\newcommand{\mfP}{\mathfrak{P}}
\newcommand{\mfQ}{\mathfrak{Q}}
\newcommand{\mcS}{\mathcal{S}}

\address{Jeremiah Horrocks Institute, University of Central Lancashire, Preston PR1 2HE, United Kingdom}
\email{sanscombe@uclan.ac.uk}
\address{Afdeling Algebra, KU Leuven, Celestijnenlaan 200b, 3001 Leuven, Belgium}
\email{philip.dittmann@kuleuven.be}
\address{Institut f\"{u}r Algebra, Technische Universit\"{a}t Dresden, 01062 Dresden, Germany}
\email{arno.fehm@tu-dresden.de}

\begin{document}

\title{Denseness results in the theory of algebraic fields}
\author{Sylvy Anscombe, Philip Dittmann and Arno Fehm}
\thanks{\today}
\maketitle

\begin{abstract}
We study when the property that a field is dense in its real and $p$-adic closures is elementary in the language of rings
and deduce that all models of the theory of algebraic fields 
have this property.
\end{abstract}

\section{Introduction}

The completions of the field of rational numbers are $\mathbb{R}$ and the $\mathbb{Q}_p$ for prime numbers $p$,
and completions of number fields are accordingly finite extensions of $\mathbb{R}$ and $\mathbb{Q}_p$.
The importance of these completions in number theory (for example in the form of local-global principles)
and the success of local algebra has motivated 
the attempt to apply these concepts also to other fields.
Here the problem arises that the completion of an ordered field or a (suitably) valued field
need in general not share any of the crucial properties of $\mathbb{R}$ or the finite extensions of the $\mathbb{Q}_p$.
This has led to the definitions of real closures and $p$-adic closures 
(of $p$-rank $1$, corresponding to $\mathbb{Q}_p$, and of higher $p$-rank corresponding to finite extensions of $\mathbb{Q}_p$)
of arbitrary fields with orderings respectively so-called $p$-valuations.
But here the converse problem arises: These closures are not complete, but do they at least satisfy other crucial properties 
of the completions - in particular, is the field dense in its closures?

This question has been answered positively for several classes of fields, in particular for fields that satisfy a local-global principle for rational points on varieties:
Prestel \cite{Prestel} shows that pseudo-real closed (PRC) fields are dense in all their real closures,
and the analogue for pseudo-$p$-adically closed (P$p$C) fields was proven in \cite{Grob}.
The most general result in this direction is by Pop \cite{Pop},
who proves that arbitrary so-called pseudo classically closed (PCC) fields are dense in all their real closures and all their $p$-adic closures (of arbitrary $p$-rank).
Ershov \cite{Ershov} has a similar result for so-called ${\rm RC}_\pi$-fields.

In this note we study when denseness in real and $p$-adic closures is an elementary property.
We show that this is not always the case (Section \ref{sec:counterexample}),
but we also prove some general principles that allow us to exhibit a large class of fields dense in all their real and $p$-adic closures,
and which likely includes all the classes listed in the previous paragraph:

\begin{Theorem}\label{thm:intro}
Every model of the theory of algebraic fields of characteristic zero
is dense in all its real closures and all its $p$-adic closures {\rm(}of arbitrary $p$-rank{\rm)}.
Equivalently (see Proposition \ref{prop:equiv_denseness_completions}), the completion of any such model with respect to any ordering is real closed, and the completion of any such model with respect to any $p$-valuation is $p$-adically closed.
\end{Theorem}
Here, by an algebraic field of characteristic zero we mean an algebraic extension of $\mathbb{Q}$.
Note that this result applies in particular to elementary extensions of number fields
and even non-standard number fields in the most general sense (see \cite{Cherlin}).
By results of Jarden \cite{Jarden, Jarden2} every PRC field and every P$p$C field is a model of the theory of algebraic fields
and this suggests that also arbitrary PCC fields might be of that kind.
The only other result about models of the theory of algebraic fields we are aware of 
concerns their cohomological dimension \cite{Chatzidakis}.
To appreciate the non-triviality of Theorem \ref{thm:intro}
note that algebraic fields are dense in their henselization with respect to {\em any} valuation (not just with respect to $p$-valuations),
which however no longer holds true for arbitrary models of the theory of algebraic fields, see Remark \ref{rem:dense_henselization}.

While Theorem \ref{thm:intro} is interesting for its own sake,
our motivation for it came from a concrete application in the area of definable valuation rings based on \cite{AnscombeFehm},
as we explain in Section \ref{sec:algebraic}. As one consequence we obtain:

\begin{Corollary}
Let $F$ be an algebraic extension of $\mathbb{Q}$ properly contained in the algebraic part of $\mathbb{R}$ or of some $\mathbb{Q}_p$.
Then there exists an existential formula $\phi(x)$ in the language of rings which defines in $F((t))$ the valuation ring $F[[t]]$.
Equivalently, there exists a polynomial $g\in\mathbb{Z}[X_1,\dots,X_n]$ such that the projection of the zero set of $g$ in $F((t))$ onto the first coordinate is precisely $F[[t]]$.
\end{Corollary}

Although Theorem \ref{thm:intro} is a result in the model theory of fields, its proof relies heavily on recent number theoretic results:
While in the case of orderings it can build on Euler's well-known Four Squares Theorem
and Siegel's generalization to number fields,
which implies the finiteness of the Pythagoras number for algebraic fields,
in the $p$-adic case we 
make use of the recent $p$-adic analogue of Siegel's theorem \cite{ADF2}
for the so-called $p$-Pythagoras number,
where the squares are replaced by a certain rational function, the Kochen operator.
The other crucial observation involved is that for algebraic fields, denseness in real or $p$-adic closures
implies a certain uniform denseness statement, which then allows an axiomatization.
In fact we show more generally that denseness in real or $p$-adic closures 
transfers to elementarily equivalent fields as soon as Pythagoras number respectively $p$-Pythagoras number are finite.
In this generality the equivalence of denseness and uniform denseness is less obvious and builds 
on our approximation theorems from \cite{ADF1}.

\section{Statement of general result}

We recall definitions from \cite{LGP} that we will use, see also \cite{PR} for basic notions regarding $p$-valuations.
Let $F$ be a field of characteristic $0$.
A {\bf prime} $\mfP$ of $F$ is either an ordering on $F$, 
in which case we also denote it by $\leq_\mfP$,
or an equivalence class of $p$-valuations on $F$ (of arbitrary $p$-rank),
in which case we write $v_\mfP$ for a representative of $\mfP$ which has $\mathbb{Z}$ as smallest non-trivial convex subgroup
of the value group.
However, for brevity, we will freely say `$\mfP$ is a $p$-valuation' to mean that $\mfP$ is an equivalence class of $p$-valuations.
We denote by $\mathcal{O}_\mfP$ the valuation ring of $v_\mfP$, if $\mfP$ is a $p$-valuation, or the positive cone of $\leq_\mfP$, if $\mfP$ is an ordering.\footnote{Note that in the case of orderings this definition agrees with \cite{LGP} but disagrees with a definition we made in \cite{ADF1}, where for an ordering $\mfP$ we put $\mathcal{O}_\mfP = \{ x \in F \colon -1 \leq_\mfP x \leq_\mfP 1 \}$.}
The set of all primes of $F$ is denoted $\mathcal{S}(F)$.
For example, we write $\mathcal{S}(\mathbb{Q})=\{2,3,\dots,\infty\}$.
A {\bf closure} of $(F,\mfP$) is a real closure if $\mfP$ is an ordering, or is a $p$-adic closure if $\mfP$ is a $p$-valuation.
Thus the closure of $(F,\mfP)$ is unique up to isomorphism over $F$ if and only if $\mfP$ is an ordering or $\mfP$ is a $p$-valuation such that $v_{\mfP}(F^\times)$ is a $\mathbb{Z}$-group
(see Section \ref{sec:Z} for the definition of a $\mathbb{Z}$-group).
Note also that every real closed field and every $p$-adically closed field carries a unique prime, 
which induces a canonical topology.

We consider two conditions on a set $S\subseteq\mathcal{S}(F)$:

\noindent
{\bf Denseness} is the condition
\begin{enumerate}
\item[${\rm D}_S$]
For every monic $g\in F[X]$ and all $\mfP\in S$
for which $g$ has a zero in some closure of $(F,\mfP)$, 
for every $a\in F^\times$ there exists $x\in F$
with $1-g(x)^2a^{-2}\in\mathcal{O}_\mfP$.
\end{enumerate}
{\bf Uniform denseness} is the condition
\begin{enumerate}
\item[${\rm UD}_S$] 
For every monic $g\in F[X]$ and every $a\in F^\times$ there exists $x\in F$
such that for all $\mfP\in S$ 
for which $g$ has a zero in some closure of $(F,\mfP)$,
we have 
$1-g(x)^2a^{-2}\in\mathcal{O}_\mfP$.
\end{enumerate}

\begin{Lemma}\label{lem:UDimpliesD}
${\rm UD}_S$ implies ${\rm D}_S$.
\end{Lemma}

\begin{proof}
This is trivial.
\end{proof}

The following lemma is a variant of \cite[Proposition 1.6.1]{Ershov_multivalued}.
For the proof we use the \emph{completion} of a field with respect to a valuation or ordering. This can be defined in terms of transfinite Cauchy sequences, see \cite[Theorem 2.4.3]{EP} respectively \cite{Hau} (written $\overline{F}_{X(F)}$ in the latter). Many alternative definitions are also available in the literature, see for instance \cite[III §6, No 8]{Bourbaki_TG} and \cite{Bae}.
\begin{Lemma}
$F$ satisfies ${\rm D}_S$ if and only if $F$ is dense in every closure
at elements of $S$.
\end{Lemma}

\begin{proof}
Let $\mfP\in S$.
First note that 
if $\mfP$ is a $p$-valuation,
then $1-g(x)^2a^{-2}\in\mathcal{O}_\mfP$ if and only if $v_\mfP(g(x))\geq v_\mfP(a)$,
and if $\mfP$ is an ordering, 
then $1-g(x)^2a^{-2}\in\mathcal{O}_\mfP$ if and only if $|g(x)|\leq_\mfP|a|$.
Now let $(F',\mfP')$ be a closure of $(F,\mfP)$.
Of course, if $F$ is dense in $F'$, then any zero $y\in F'$ of $g\in F[X]$ can be approximated
arbitrarily well by elements $x\in F$.

Conversely, we suppose that $F$ satisfies ${\rm D}_{S}$.
If an irreducible polynomial $g\in F[X]$ has a zero $y$ in $F'$ then we may find $x\in F$ with $v_\mfP(g(x))$ arbitrarily large if $\mfP$ is a $p$-valuation, or with $|g(x)|$ arbitrarily small if $\mfP$ is an ordering.
Then $g$ has a zero in the completion 
$\hat{F}$ of $F$ with respect to $\mfP$: this follows from \cite[Theorem 2.4.5]{EP} in the case of valuations, and is proved analogously in the case of orderings, by factoring $g$ into linear and quadratic factors over a real closure of $\hat{F}$.
Since this holds for all such $g$, and $F'/F$ is algebraic, there is an $F$-embedding of $F'$ into $\hat{F}$, see~\cite[Lemma 20.6.3(a)]{FJ}.
As $F'$ carries a unique prime, this embedding is in fact an isometry, and so $F$ is dense also in $F'$. 
\end{proof}

The main idea of the proof of the generalization of Theorem \ref{thm:intro} is that for sufficiently nice sets of primes $S$,
${\rm D}_S$ in fact implies ${\rm UD}_S$, which in turn is an elementary statement given a predicate for the holomorphy domain:

Fix a number field $K$ and a pair $\tau=(e,f)\in\mathbb{N}^{2}$, where $\mathbb{N}$ is the set of positive integers.
For a prime $\mfp$ of $K$ and an extension $F/K$, we denote by $\mcS_\mfp^{\tau}(F)$ the set of primes of $F$ lying above $\mfp$,
where in the case of $p$-valuations we restrict to 
relative initial ramification $e'$ at most $e$ and relative residue degree $f'$ dividing $f$.
The pair $\tau'=(e',f')$ is called the {\bf relative type} of $\mfP$, and we write $\tau'\leq\tau$ if and only if
$e'\leq e$ and $f'|f$.
The {\bf holomorphy domain} corresponding to $\mathcal{S}_\mfp^\tau(F)$ is 
\begin{align*}
     R_\mfp^\tau(F) = \bigcap_{\mfP\in\mathcal{S}_\mfp^\tau(F)}\mathcal{O}_\mfP.
\end{align*}
We adopt the convention that the relative type $\tau=(1,1)$ can always be omitted from any notation.
We write $\mathcal{S}_\mfp^{\tau,\mathbb{Z}}(F)$ for the set of those $\mfP\in\mathcal{S}_\mfp^\tau(F)$ that are either orderings, or are $p$-valuations for which $v_{\mfP}(F^\times)$ is a $\mathbb{Z}$-group.
We write $\mathcal{S}_\mfp^{=\tau}(F)$ for the set of those $\mfP\in\mathcal{S}_\mfp^\tau(F)$ that are either orderings, or are $p$-valuations of exact relative type $\tau$.
For $t,s\in F$ we denote by $\mathcal{S}_\mfp^{=\tau}(F,t,s)$ the set of those $\mfP\in\mathcal{S}_\mfp^{=\tau}(F)$ 
that are either orderings, or are $p$-valuations
for which $t$ is a uniformizer and $s$ is a unit whose residue generates the multiplicative group of the residue field.
Finally, we let $\mathcal{S}_\mfp^{=\tau,\mathbb{Z}}(F,t,s)=\mathcal{S}_\mfp^{=\tau}(F,t,s)\cap\mathcal{S}_\mfp^{\tau,\mathbb{Z}}(F)$.

To state the general results
we denote by $\mathcal{L}_{\rm ring}=\{+,\cdot,0,1\}$ the language of rings
and
we define $\mathcal{L}_R=\mathcal{L}_{\rm ring}\cup\{R\}$, where $R$ is a unary predicate,
as well as expansions $\mathcal{L}_{\rm ring}(K)$ and $\mathcal{L}_R(K)$ by constants for the elements of $K$.
Then the following holds:

\begin{Proposition}\label{prop:DimpliesUD}
${\rm D}_{\mathcal{S}_\mfp^\tau(F)}$ holds if and only if ${\rm UD}_{\mathcal{S}_\mfp^{=\tau'}(F,t,s)}$ holds for every $t,s\in F$ and every $\tau'\leq\tau$.
\end{Proposition}

\begin{Proposition}\label{prop:axiom}
There exists an $\mathcal{L}_R(K)$-theory $T_\mfp^\tau$ such that
$(F,R_\mfp^\tau(F))\models T_\mfp^\tau$ if and only if
${\rm UD}_{\mathcal{S}_\mfp^{=\tau'}(F,t,s)}$ holds for every $t,s\in F$ and every $\tau'\leq\tau$.
\end{Proposition}

\begin{Corollary}\label{cor:Tp}
$F$ is dense in every closure at elements of $\mathcal{S}_\mfp^\tau(F)$
if and only if $(F,R_\mfp^\tau(F))\models T_\mfp^\tau$. 
\end{Corollary}

We prove Proposition \ref{prop:DimpliesUD} in Section \ref{sec:DimpliesUD} and Proposition \ref{prop:axiom} in Section \ref{sec:axiom}.
In the case of algebraic fields (the setting of Theorem \ref{thm:intro}) we obtain an $\mathcal{L}_{\rm ring}(K)$-theory,
since the holomorphy domain is uniformly definable.
This part is carried out in Section \ref{sec:algebraic},
where we also draw some consequences of the results.
In Section \ref{sec:counterexample} we show that 
the $\mathcal{L}_R$-theory $T_\infty$
cannot be replaced by an $\mathcal{L}_{\rm ring}$-theory
with the same property.
We start in Section \ref{sec:independence} with some purely algebraic characterizations of denseness conditions.

\section{Denseness in Henselizations, $p$-adic Closures and Real Closures}
\label{sec:independence}

Let $F$ be a field.
For a valuation $v$ on $F$, we write $\mathcal{O}_{v}$ for its valuation ring, $vF$ for its value group, and $Fv$ for its residue field.
A pair of valuations on $F$ is {\bf independent} if they induce distinct topologies on $F$.
Likewise for a pair of orderings,
or a pair consisting of an ordering and a valuation.
In particular, this applies to primes of $F$.
A valuation $u$ on $F$ is a {\bf refinement} of $v$, and $v$ is a {\bf coarsening} of $u$, if $\mathcal{O}_{u}\subseteq\mathcal{O}_{v}$.
If $u$ is a coarsening or refinement of $v$, then $u$ and $v$ are {\bf comparable}.
Note that a $p$-valuation and a $q$-valuation,
for primes $p$ and $q$, are never comparable
as soon as they are distinct.

\begin{Lemma}\label{lem:Ershov}\label{lem:Dense_in_henselization_polynomials}
Let $v$ be a valuation on $F$.
The following are equivalent:
\begin{enumerate}
\item $F$ is dense in a henselization of $(F,v)$ {\rm(}with respect to the valuation topology{\rm)}
\item For every non-trivial coarsening $w$ of $v$,
the valuation $\bar{v}$ induced by $v$ on the residue field $Fw$ is henselian.
\item For any separable monic polynomial $g \in \mathcal{O}_v[X]$ for which there exists $a_0 \in \mathcal{O}_v$ with $v(g(a_0)) > 0$ and $v(g'(a_0)) = 0$, the set $\{ v(g(a)) \colon a \in F, g(a) \neq 0 \}$ is unbounded in $vF$.
\end{enumerate}
\end{Lemma}

\begin{proof}
See
\cite[Theorem 1.6.1]{Ershov_multivalued} and
\cite[Proposition 1.6.1]{Ershov_multivalued}.
\end{proof}

\begin{Lemma}\label{lem:denseness_preserves_val_grp_res_fld}
  Let $F'$ be a field extension of $F$ with a valuation $v'$ extending $v$. 
  If $F$ is dense in $(F',v')$, then $F'v' = Fv$ and $v'F' = vF$.
\end{Lemma}
\begin{proof}
  Given $x \in \mathcal{O}_{v'}$, any $y \in F$ with $v'(x-y) > 0$ has the same residue as $x$; this proves the equality of residue fields.
  Given $x \in(F')^{\times}$, any $y \in F$ with $v'(x-y) > v'(x)$ has the same valuation as $x$; this proves the equality of value groups.
\end{proof}

\begin{Lemma}\label{lem:trichotomy}
Let $v_{1},v_{2}$ be valuations on $F$.
Suppose that $F$ is dense in henselizations of both $(F,v_{1})$ and $(F,v_{2})$.
Then
\begin{enumerate}
\item $v_{1}$ and $v_{2}$ are comparable, or
\item $v_{1}$ and $v_{2}$ are independent, or
\item both $Fv_{1}$ and $Fv_{2}$ are algebraically closed.
\end{enumerate}
\end{Lemma}

\begin{proof}
Suppose that $v_{1}$ and $v_{2}$ are incomparable and dependent.
Then both $v_{1}$ and $v_{2}$ are non-trivial.
Let $w$ be the finest common coarsening of $v_{1}$ and $v_{2}$.
Since $v_{1}$ and $v_{2}$ are dependent,
$w$ is also non-trivial, see \cite[Theorem 2.3.4]{EP}.
Since $v_{1}$ and $v_{2}$ are incomparable,
$w$ is a proper coarsening of both $v_{1}$ and $v_{2}$.
Thus $v_{1}$ and $v_{2}$ induce independent non-trivial valuations $\bar{v}_{1}$ and $\bar{v}_{2}$ on $Fw$.
By $(1)\Rightarrow(2)$ of Lemma \ref{lem:Ershov}, both $\bar{v}_{1}$ and $\bar{v}_{2}$ are henselian.
Therefore, by F.\ K.\ Schmidt's theorem \cite[Theorem 4.4.1]{EP}, $Fw$ is separably closed,
and so both $Fv_{1}$ and $Fv_{2}$ are algebraically closed,
see \cite[Theorem 3.2.11]{EP}.
\end{proof}

Let $F$ now be a field of characteristic $0$.
Given a valuation $u$ and an ordering $\leq$ on $F$, we say that $u$ is {\bf $\leq$-convex} if its valuation ring $\mathcal{O}_{u}$ is a convex subset of $F$ with respect to $\leq$.
We may now say that $u$ is a {\bf coarsening} of a prime $\mfP\in\mathcal{S}(F)$ if
$u$ is a coarsening of $v_{\mfP}$, in the usual sense for valuations, if $\mfP$ is a $p$-valuation, or if $u$ is $\leq_{\mfP}$-convex, if $\mfP$ is an ordering.

Given an ordering $\leq$ on $F$,
we denote by $v_{\leq}$ the valuation whose valuation ring is the convex hull of the prime field $\mathbb{Q}$ with respect to $\leq$.
Thus
$v_{\leq}$ is the finest valuation which is $\leq$-convex,
and
$u$ is $\leq$-convex if and only if $u$ is a coarsening of $v_{\leq}$.
Note that 
$\leq$ is archimedean if and only if $v_{\leq}$ is the trivial valuation.
Also note that
if $\leq$ is non-archimedean then $\leq$ and $v_{\leq}$ are dependent,
whereas
if $\leq$ is archimedean then it is independent from all valuations and all other orderings.

If $u$ is $\leq$-convex,
then $\leq$ induces an ordering on $Fu$, which we denote by $\leq_{u}$.
In general, $\leq_u$ does not uniquely determine $\leq$:
by the Baer--Krull theorem \cite[Theorem 2.2.5]{EP}, there is a {family} of orderings on $F$ each of which induces $\leq_{u}$ on $Fu$.

We can now give an equivalent statement of denseness conditions in terms of completions.
\begin{Proposition}\label{prop:equiv_denseness_completions}
  Let $F$ be a field.
  \begin{enumerate}
  \item If $v$ is a valuation on $F$ and $(F', v')$ a henselization of $(F,v)$, then $F$ is dense in $(F', v')$ if and only if the completion $(\hat F, \hat v)$ of $(F,v)$ is henselian.
  \item If $v$ is a $p$-valuation on $F$ and $(F', v')$ a $p$-adic closure of $(F, v)$, then $F$ is dense in $(F', v')$ if and only if the completion $(\hat F, \hat v)$ of $(F,v)$ is $p$-adically closed.
  \item If $\leq$ is an ordering on $F$ and $(F', \leq')$ a real closure of $(F, \leq)$, then $F$ is dense in $(F', \leq')$ if and only if the completion $(\hat F, \hat\leq)$ of $(F,v)$ is real closed.
  \end{enumerate}
\end{Proposition}

\begin{proof}
  For the first point, observe that if the completion $(\hat F, \hat v)$ is henselian, then the henselization $(F', v')$ embeds into $\hat{F}$, and therefore $F$ is dense in it.
  For the converse direction we imitate an argument from the proof of \cite[Proposition 1.6.1]{Ershov_multivalued}: if $(F, v)$ is dense in its henselization, then $(\hat F, \hat v)$ is in turn dense in its own henselization, as the condition from Lemma \ref{lem:Dense_in_henselization_polynomials}(3) transfers from $(F, v)$ to $(\hat F, \hat v)$ by approximating a separable polynomial in $\mathcal{O}_{\hat v}[X]$ by a separable polynomial in $\mathcal{O}_v[X]$.
  By completeness, this means that $(\hat F, \hat v)$ must already be henselian.

  The second point follows from the first by observing that by Lemma \ref{lem:denseness_preserves_val_grp_res_fld} $F$ can only be dense in a $p$-adic closure if $vF$ is a $\mathbb{Z}$-group, in which case the henselization of $(F, v)$ is already a $p$-adic closure.

  The third point is immediate from \cite[Satz 13, Satz 14]{Hau}.  
\end{proof}

For lack of a convenient reference, we give a proof of the following simple lemma.

\begin{Lemma}\label{lem:compatible_prolongation}
Let $F'/F$ be any field extension.
\begin{enumerate}
\item
If $v'$ is a valuation on $F'$ prolonging a valuation $v$ on $F$,
and if $w$ is a coarsening of $v$,
then there is a coarsening $w'$ of $v'$ which is a prolongation of $w$.
\item
If $\leq'$ is an ordering on $F'$ extending an ordering $\leq$ on $F$,
and if $w$ is a coarsening of $\leq$,
then there is a coarsening $w'$ of $\leq'$ which is a prolongation of $w$.
\end{enumerate}
\end{Lemma}
\begin{proof}
Recalling the correspondence between coarsenings of $v$ and convex subgroups of $vF$,
we denote by $\Delta$ the subgroup corresponding to $w$.
In the analogous correspondence for coarsenings of $v'$,
we choose $w'$ to correspond to the convex hull of $\Delta$ in $v'F'$.
Then $w'$ is both a coarsening of $v'$ and a prolongation of $w$.
This proves (1).
For (2), we apply (1) in the case $v=v_{\leq}$ and $v'=v_{\leq'}$.
\end{proof}

\begin{Lemma}\label{lem:EP.4}
Let $\leq$ be an ordering on $F$ and let $v$ be a $\leq$-convex valuation.
Then $(F,\leq)$ is real closed if and only if
\begin{enumerate}
\item $vF$ is divisible,
\item $(Fv,\leq_{v})$ is real closed, and
\item $(F,v)$ is henselian.
\end{enumerate}
\end{Lemma}
\begin{proof}
See
{\cite[Theorem 4.3.7]{EP}}.
\end{proof}

\begin{Lemma}\label{lem:orderings.A}
Let $\leq$ be an ordering on $F$ and let $v$ be a non-trivial $\leq$-convex valuation.
Then $F$ is dense in a real closure of $(F,\leq)$ if and only if
\begin{enumerate}
\item $vF$ is divisible,
\item $(Fv,\leq_{v})$ is real closed, and
\item $F$ is dense in a henselization of $(F,v)$.
\end{enumerate}
\end{Lemma}
\begin{proof}
$(\Longrightarrow)$
Let $(F',\leq')$ be a real closure of $(F,\leq)$,
and let $(F^{h},v^{h})$ be a henselization of $(F,v)$.
By Lemma \ref{lem:compatible_prolongation}(2),
there is a valuation $v'$ on $F'$ which is a prolongation of $v$ and which is a coarsening of $v_{\leq'}$,
i.e.~$v'$ is $\leq'$-convex.
As $F$ is dense in $(F',\leq')$ and $v'$ induces the same topology as $\leq'$, $F$ is also dense in $(F',v')$, in particular (by Lemmas \ref{lem:denseness_preserves_val_grp_res_fld} and \ref{lem:EP.4})
$Fv=F'v'$ is real closed and $vF=v'F'$ is divisible.
Also by Lemma \ref{lem:EP.4},
$(F',v')$ is henselian,
hence $(F^{h},v^{h})$ embeds into $(F',v')$ over $(F,v)$,
so $F$ is dense also in $(F^{h},v^{h})$.

$(\Longleftarrow)$
By the Baer--Krull theorem,
and since $vF$ is divisible,
$\leq$ is the unique ordering on $F$ which induces $\leq_{v}$ on $Fv$.
Any henselization $(F^{h},v^{h})$ of $(F,v)$
is an immediate extension,
hence the value group $v^{h}F^{h}=vF$ is still divisible,
and the residue field is still $Fv$.
By the Baer--Krull theorem, there is a unique ordering $\leq^{h}$ on $F^{h}$ which induces $\leq_{v}$ on $Fv$
and obviously extends $\leq$.
As $\leq^h$ and $v^h$ induce the same topology on $F^h$,
$F$ is dense in $(F^h,\leq^h)$.
Applying Lemma \ref{lem:EP.4} and assumptions (1) and (2), in fact $(F^{h},\leq^{h})$ is real closed,
hence is a real closure of $(F,\leq)$.
\end{proof}

The next proposition is an analogue of Lemma \ref{lem:Ershov}(2) for non-archimedean orderings.

\begin{Proposition}\label{prp:dense.in.real.closure}
Let $\leq$ be an ordering on $F$.
Then $F$ is dense in a real closure of $(F,\leq)$ if and only if
for all non-trivial $\leq$-convex valuations $v$ on $F$ we have
\begin{enumerate}
\item $vF$ is divisible and
\item $(Fv,\leq_{v})$ is real closed.
\end{enumerate}
\end{Proposition}
\begin{proof}
If $F$ is dense in a real closure of $(F,\leq)$, then (1) and (2) follow from Lemma \ref{lem:orderings.A}.
For the converse,
we suppose (1) and (2) hold.
First observe that if $\leq$ is archimedean then $F$ is dense in every real closure of $(F,\leq)$.
Thus we may suppose $\leq$ to be non-archimedean, and so $v_{\leq}$ is non-trivial.
If $w$ is any non-trivial $\leq$-convex valuation on $F$ (equivalently, any non-trivial coarsening of $v_{\leq}$),
then
$(Fw,\leq_{w})$ is real closed, by (2).
In particular $(Fw,\overline{v_{\leq}})$ is henselian,
by Lemma \ref{lem:EP.4}.
Since $w$ was an arbitrary non-trivial coarsening of $v_{\leq}$,
it follows from $(2)\Rightarrow(1)$ of Lemma \ref{lem:Ershov} that $F$ is dense in a henselization of $(F,v_{\leq})$.
This verifies condition (3) of Lemma \ref{lem:orderings.A} for $v=v_{\leq}$.
Therefore $F$ is dense in a real closure of $(F,\leq)$, as required.
\end{proof}

\begin{Proposition}\label{lem:D_implies_independence}
Let $S\subseteq\mathcal{S}(F)$.
If ${\rm D}_S$ holds then the elements of $S$ are pairwise independent.
\end{Proposition}

\begin{proof}
Let $\mfP$ and $\mfQ$ be distinct primes in $S$, and suppose that $F$ is dense in every closure of
$(F,\mfP)$ and
$(F,\mfQ)$.
Note that if $\mfP$ is a $p$-valuation, this implies in particular
that $F$ is dense in a henselization of $(F,v_\mfP)$.

As a first case, suppose that $\mfP$ is a $p$-valuation
and $\mfQ$ is a $q$-valuation, for prime numbers $p$ and $q$.
Since such valuations are incomparable, and the residue field of a $p$-valuation is finite, by Lemma \ref{lem:trichotomy} we have that $\mfP$ and $\mfQ$ are independent.

Next, suppose that $\mfP$ is an ordering and $\mfQ$ is a $p$-valuation.
If $\leq_{\mfP}$ is archimedean, then it is independent from all valuations and other orderings;
in particular it is independent from $v_{\mfQ}$.
If $\leq_{\mfP}$ is non-archimedean, then the finest $\leq_{\mfP}$-convex valuation
$v_{1}=v_{\leq_{\mfP}}$ is non-trivial.
By Lemma \ref{lem:orderings.A}, $F$ is dense in a henselization of $(F,v_{1})$,
$v_{1}F$ is divisible,
and
$(Fv_{1},\leq_{1})$ is real closed,
where $\leq_{1}$ is the ordering induced on $Fv_{1}$ by $\leq_{\mfP}$.
Since a real closed field admits no $p$-valuations, $v_{1}$ and $v_{\mfQ}$ are incomparable.
Thus, by Lemma \ref{lem:trichotomy},
$v_{1}$ and $v_{\mfQ}$ are independent,
and it follows that
$\mfP$ and $\mfQ$ are independent.

Finally, suppose that $\mfP$ and $\mfQ$ are both orderings.
If either is archimedean, then they are independent;
thus we may assume that neither is archimedean,
and so both $v_{1}=v_{\leq_{\mfP}}$ and $v_{2}=v_{\leq_{\mfQ}}$ are non-trivial.
By Lemma \ref{lem:orderings.A},
$F$ is dense in henselizations of $(F,v_{1})$ and $(F,v_{2})$,
both $v_{1}F$ and $v_{2}F$ are divisible,
and both $(Fv_{1},\leq_{1})$ and $(Fv_{2},\leq_{2})$ are real closed,
where $\leq_{1}$ is the ordering induced on $Fv_{1}$ by $\leq_{\mfP}$, and likewise $\leq_{2}$ is induced on $Fv_{2}$ by $\leq_{\mfQ}$.
We claim that $v_{1}$ and $v_{2}$ are incomparable.
To see this:
suppose without loss of generality that
$v_{1}$ is a (not necessarily proper) refinement of $v_{2}$.
Then $v_{2}$ is $\leq_{\mfP}$-convex, and so $\leq_{\mfP}$ induces an ordering $\leq$ on $Fv_{2}$.
Since the real closed field $Fv_{2}$ admits a unique ordering, we must have $\leq=\leq_{2}$.
Since $v_{2}F$ is divisible, by the Baer--Krull theorem,
we have $\leq_{\mfP}=\leq_{\mfQ}$, i.e.\ $\mfP=\mfQ$,
which is a contradiction, proving the claim.
Since neither $Fv_{1}$ nor $Fv_{2}$ is algebraically closed,
and together with the claim,
it follows from Lemma \ref{lem:trichotomy} that $v_{1}$ and $v_{2}$ are independent,
and therefore $\mfP$ and $\mfQ$ are independent, as required.
\end{proof}

\section{Uniform denseness}
\label{sec:DimpliesUD}

Let $K$ be a number field, $\mfp$ a prime of $K$, $\tau=(e,f)\in\mathbb{N}^{2}$ and $F$ an extension of $K$.
If $\mfp$ is a $p$-valuation let $t_\mfp$ be a uniformizer of $v_\mfp$, 
and let $q_{\mfp}=|Kv_{\mfp}|$ be the size of the residue field.
We equip the set $\mathcal{S}(F)$ of all primes of $F$ with the topology with subbasis of open-closed sets
$$
 \{\mfP\in\mathcal{S}(F) : a\in\mathcal{O}_\mfP \}, \quad a\in F.
 $$
This clearly makes $\mathcal{S}(F)$ a totally disconnected Hausdorff space.

\begin{Lemma}\label{lem:SFcompact}
$\mathcal{S}_\mfp^\tau(F)$ is compact.
\end{Lemma}

\begin{proof}
\cite[Lemma 10.3]{LGP} states that the subspace topology on $\mathcal{S}_\mfp^\tau(F)$ coincides with the subspace topology obtained by declaring the $\{\mfP\in\mathcal{S}(F) : a\in\mathcal{O}_\mfP \}$ to be a subbasis of open sets, and that this gives a compact space.
\end{proof}

\begin{Lemma}\label{lem:SFsw}
Suppose that $\mfp$ is a $p$-valuation, let $\mfP\in\mathcal{S}_\mfp^\tau(F)$ and let $t,s\in F$.
Then $\mfP\in\mathcal{S}_\mfp^{=\tau}(F,t,s)$ if and only if
$t^et_\mfp^{-1}\in\mathcal{O}_\mfP^\times$, $s\in\mathcal{O}_\mfP^\times$
and 
$s^n - 1 \in \mathcal{O}_\mfP^\times$ for all $n | q_\mfp^f - 1$, $n \neq q_\mfp^f - 1$.
\end{Lemma}

\begin{proof}
This follows directly from the definitions.
\end{proof}

\begin{Lemma}\label{lem:SFsw_clopen}
If $\mfp$ is a $p$-valuation, then
for every $t,s\in F$, $\mathcal{S}_\mfp^{=\tau}(F,t,s)$ is open-closed in $\mathcal{S}_\mfp^\tau(F)$.
\end{Lemma}

\begin{proof}
As each of the finitely many conditions in the previous lemma defines an open-closed subset of $\mathcal{S}_\mfp^\tau(F)$,
also $\mathcal{S}_\mfp^{=\tau}(F,t,s)$ is open-closed.
\end{proof}

Next we quote two approximation results from \cite{ADF1}:

\begin{Proposition}\label{prop:value_approximation}
Suppose that $\mfp$ is a $p$-valuation.
  Let $S_1, \dotsc, S_n \subseteq \mathcal{S}_\mfp^\tau(F)$ be disjoint compact sets, and $z_1, \dotsc, z_n \in F^\times$.
  If for every valuation $v$ on $F$ coarsening some prime in $S_i$ and some prime in $S_j$, for some $i,j$, we have $v(z_i) = v(z_j)$, then there exists a $z \in F^\times$ with $v_\mfP(z) = v_\mfP(z_i)$ for all $\mfP \in S_i$, for $i = 1, \dotsc, n$.
\end{Proposition}
\begin{proof}
This is an application of \cite[Corollary 5.1]{ADF1}, for which we verify the hypotheses.
Condition (U) holds as shown in \cite[Example 3.5(2,3)]{ADF1}.
Condition (I) follows from the fact that distinct elements of $\mathcal{S}_{\mfp}^{\tau}(F)$ are pairwise incomparable.
\end{proof}

\newcommand\B{\mathrm{B}}
For the next proposition we will need the notion of a ball with respect to a prime $\mfP$ of $F$: if $\mfP$ is a $p$-valuation we write
$$
    \B_{\mfP}(y,z):=\{x\in F:v_{\mfP}(x-y)>v_{\mfP}(z)\},
$$
whereas if $\mfP$ is an ordering we write
$$
    \B_{\mfP}(y,z):=\{x\in F:|x-y|<_{\mfP}|z|\},
$$
for $y\in F$ and $z\in F^{\times}$.

\begin{Proposition}\label{prop:approxfunctions}
Let
$S\subseteq\mathcal{S}_{\mfp}^{\tau}(F)$ be compact,
and suppose that the elements of
$S$
are pairwise independent.
Let $\gamma_1,\dots,\gamma_m\in F(X_1,\dots,X_r)$ and write $\gamma_j=\frac{g_j}{h_j}$ with $g_j,h_j\in F[X_1,\dots,X_r]$.
For each $j=1,\dots,m$, let $y_{j}\in F$, $z_{j}\in F^\times$ be given.
If for each $\mfP\in S$
there exists $x_\mfP\in F^r$ such that 
$h_j(x_\mfP)\neq 0$ and
$\gamma_j(x_\mfP)\in\B_\mfP(y_{j},z_{j})$ for each $j$,
then there exists $x\in F^r$ such that
$h_j(x)\neq 0$ and $\gamma_j(x)\in\B_\mfP(y_{j},z_{j})$ for each $j$ and $\mfP\in S$.
\end{Proposition}
\begin{proof}
This is an application of \cite[Corollary 7.8]{ADF1}.
Note that $\mathcal{S}_{\mfp}^{\tau}(F)$ satisfies 
condition (U) of that corollary by \cite[Example 3.5(2,3)]{ADF1}.
\end{proof}

For a set $S\subseteq\mathcal{S}(F)$ of primes of $F$ and an
$\mathcal{L}_R(F)$-sentence $\varphi$ we define
$$
 \varphi(S) := \{\mfP\in S:(F,\mathcal{O}_\mfP)\models\varphi \}
$$
and
$$
 \varphi'(S) := \{\mfP\in S:(F',\mathcal{O}_{\mathfrak{P}'})\models\varphi\mbox{ for each closure }(F',\mfP')\mbox{ of }(F,\mfP)\}.
$$

\begin{Lemma}\label{lem:local_condition_qf}
If $\varphi$ is a quantifier-free $\mathcal{L}_R(F)$-sentence and $S\subseteq\mathcal{S}(F)$, 
then $\varphi(S)=\varphi'(S)$ is open-closed in $S$.
\end{Lemma}

\begin{proof}
Since the family of open-closed sets is closed under boolean operations we can reduce
to atomic sentences $a=0$ and $R(a)$, with $a\in F$.
For $a=0$ the statement is trivial,
while for $R(a)$ it follows from the definition of the topology on $\mathcal{S}(F)$.
\end{proof}

\begin{Lemma}\label{prop:universal}
There exists a recursive map $\varphi\mapsto\varphi_\mfp^\tau$ from $\mathcal{L}_R(K)$-sentences to universal $\mathcal{L}_R(K)$-sentences
such that for every $\mfP\in\mathcal{S}_\mfp^{=\tau,\mathbb{Z}}(F)$ and every closure $(F',\mfP')$ of $(F,\mfP)$ the following are equivalent:
\begin{enumerate}
\item $(F',\mathcal{O}_{\mfP'})\models\varphi$
\item $(F',\mathcal{O}_{\mfP'})\models\varphi_\mfp^\tau$
\item $(F,\mathcal{O}_\mfP)\models\varphi_\mfp^\tau$ 
\end{enumerate}
\end{Lemma}

\begin{proof}
This is a consequence of model completeness, see \cite[Prop.~8.2]{LGP}.
\end{proof}

\begin{Proposition}\label{prop:local_condition_clopen}
Let $S \subseteq \mathcal{S}_\mfp^{=\tau,\mathbb{Z}}(F)$.
Then $\varphi'(S)$ is open-closed in $S$
for every $\mathcal{L}_R(F)$-sentence $\varphi$.
\end{Proposition}

\begin{proof}
By Lemma \ref{prop:universal}
we can find a quantifier-free $\mathcal{L}_R(F)$-formula $\psi_1(x)$, $x=(x_1,\dots,x_r)$,  such that with $\varphi_1\equiv \forall x\psi_1(x)$
we have $\varphi'(S)=\varphi_1'(S)=\varphi_1(S)$.
Thus,
\begin{align*}
	\varphi'(S) &= \bigcap_{a\in F^r}\psi_1(a)(S).
\end{align*}
Similarly, we find a quantifier-free $\mathcal{L}_R(F)$-formula $\psi_2(x)$ with
$$
 S\setminus\varphi'(S) = (\neg\varphi)'(S) = \bigcap_{a\in F^r}\psi_2(a)(S).
$$ 
By Lemma \ref{lem:local_condition_qf}, each $\psi_i(a)(S)$ is in particular closed,
so $\varphi'(S)$ is open-closed.
\end{proof}

We are now able to conclude that denseness 
implies uniform denseness, for certain compact sets of primes.

\begin{Proposition}\label{prop:D_implies_UD}
For every compact subset $S\subseteq\mathcal{S}_\mfp^{=\tau}(F)$,
${\rm D}_S$ implies ${\rm UD}_S$. 
\end{Proposition}

\begin{proof}
Suppose that ${\rm D}_{S}$ holds.
Then $S$ is contained in $\mathcal{S}_\mfp^{=\tau,\mathbb{Z}}(F)$
by Lemma \ref{lem:denseness_preserves_val_grp_res_fld},
and the elements of $S$ are pairwise independent by
Proposition \ref{lem:D_implies_independence}.
Let $g\in F[X]$ monic and $a\in F^\times$.
By Proposition \ref{prop:local_condition_clopen}, 
the set $S_g$ of $\mfP\in S$ 
for which $g$ has a zero in some (or equivalently, every) closure of $(F,\mfP)$ is open-closed in $S$,
therefore in particular compact.
For each $\mfP\in S_g$, ${\rm D}_{S_g}$ implies that there exists $x_\mfP\in F$ with $1-g(x_\mfP)^2a^{-2}\in\mathcal{O}_\mfP$.
Therefore, Proposition \ref{prop:approxfunctions} (in the case $m=r=1$) gives $x\in F$ with $1-g(x)^2a^{-2}\in\mathcal{O}_\mfP$ for every $\mfP\in S_g$.
\end{proof}

\begin{Remark}\label{rem:examples_compact}
Important special cases to which the previous proposition applies are:
\begin{enumerate}
\item $S=\mathcal{S}_\mfp^{=\tau}(F,t,s)$ for $t,s\in F$, which is compact by Lemma \ref{lem:SFsw_clopen} and Lemma \ref{lem:SFcompact},
\item $S=\mathcal{S}_\mfp(F)$, as $\mathcal{S}_\mfp(F)=\mathcal{S}_\mfp^{=(1,1)}(F,t_\mfp,1)$.
\end{enumerate}
\end{Remark}

\begin{Remark}
We want to stress that we proved Proposition \ref{prop:D_implies_UD} only for 
compact sets $S$.
This condition cannot be dropped.
More precisely, if ${\rm UD}_S$ holds for some $S\subseteq\mathcal{S}_\mfp^{=\tau}(F)$,
then $S$ is relatively compact in $\mathcal{S}_\mfp^{=\tau}(F)$ :

Indeed, by applying ${\rm UD}_S$ to the cyclotomic polynomial
$g=\Phi_{q^f-1}$ and $a=t_\mfp$
we find $s\in F$ such that the residue of $s$ is a generator of the multiplicative group of the residue field of $(F,\mfP)$,
for every $\mfP\in S$.
Next, adapting a familiar argument (for example in \cite[Lemma 3.5(ii)]{PR}),
we may select finitely many polynomials $g_{1},\dots,g_{n}\in K(s)[X]$ in order that
for each $\mfP\in S$ there exists $i\in\{1,\dots,n\}$ such that
there is a uniformizer of $(F',\mfP')$ with minimal polynomial $g_i$ over $K'(s)$,
where $(F',\mfP')$ and $(K',\mfp')$ are closures of $(F,\mfP)$ and $(K,\mfp)$, respectively.
Moreover, any element of $F'$ sufficiently close to a root of $g_{i}$ is a uniformizer of $\mfP'$.
Applying ${\rm UD}_{S}$ to each $g_{i}$,
there are elements $t_{i}\in F$ such that 
$S$ is contained in the finite union $\bigcup_{i=1}^{n}\mathcal{S}_\mfp^{=\tau}(F,t_{i},s)$,
which is compact by Remark \ref{rem:examples_compact}.
\end{Remark}

The following example shows that 
 $\mathcal{S}_\mfp^{=\tau}(F)$ itself is in general not compact.

\begin{Example}
Fix $K=\mathbb{Q}$ and a prime number $p$.
We start with a quadratic number field $K_0$ in which $p$ is split. 
Now in each step construct $K_{i+1}$ as a quadratic extension of $K_i$ which satisfies the following three conditions:
\begin{enumerate}
\item one of the two primes above $p$ of type $(1,1)$ is inert,
\item the other prime of type $(1,1)$ is split, and
\item all primes above $p$ of type $(1,2)$ are split.
\end{enumerate}
The existence of such $K_{i+1}$ is a special case of the Grunwald--Wang theorem.
If $F$ is the union of the $K_i$, then the $p$-valuations on $F$ correspond to branches in the tree of primes constructed
and are all in $S_F:=\mathcal{S}_p^{(1,2)}(F)$.
Out of these, exactly one will have type $(1,1)$, and all the others have type $\tau:=(1,2)$.
Now $S_F':=\mathcal{S}_p^{=\tau}(F)$ is not compact,
as the open-closed cover $S_F'=\bigcup_{s\in F}\mathcal{S}_p^{=\tau}(F,p,s)$ has no finite subcover:
If $S_F'=\bigcup_{j=1}^n\mathcal{S}_p^{=\tau}(F,p,s_j)$ then there exists $i$ with $s_j\in K_i$ for all $j$,
and there are $\mfP\in S_F$ of type $\tau$ whose restriction to $K_i$ is of type $(1,1)$,
so $\mfP\notin\mathcal{S}_p^{=\tau}(F,p,s_j)$ for all $j$.
Assuming that $p$ is odd for simplicity, one can in fact check that $F$ satisfies ${\rm D}_{S_F}$ but not ${\rm UD}_{S_F}$
for the polynomial $g=X^2-c$, where $c$ is an integer which is not a square modulo $p$.
\end{Example}

In this situation we have
$\mathcal{S}_{\mfp}^{\tau,\mathbb{Z}}(F)=\mathcal{S}_{\mfp}^{\tau}(F)$,
which additionally shows that 
Proposition \ref{prop:local_condition_clopen}
does not hold for $\mathcal{S}_\mfp^{\tau,\mathbb{Z}}(F)$ instead of $\mathcal{S}_\mfp^{=\tau,\mathbb{Z}}(F)$ (take $\varphi$ to be an $\mathcal{L}_R$-sentence asserting that the residue field has $p^2$ elements).

\begin{proof}[Proof of Proposition \ref{prop:DimpliesUD}]
For every $t,s\in F$ and $\tau'\leq\tau$,
the set $\mathcal{S}_\mfp^{=\tau'}(F,t,s)$
is compact by Remark \ref{rem:examples_compact},
so ${\rm D}_{\mathcal{S}_\mfp^{\tau}(F)}$ implies ${\rm UD}_{\mathcal{S}_\mfp^{=\tau'}(F,t,s)}$  by Proposition \ref{prop:D_implies_UD}.
Conversely, 
$$
 \mathcal{S}_\mfp^\tau(F)=\bigcup_{\tau'\leq\tau}\bigcup_{t,s\in F}\mathcal{S}_\mfp^{=\tau'}(F,t,s),
$$ 
and ${\rm UD}_{\mathcal{S}_\mfp^{=\tau'}(F,t,s)}$ implies ${\rm D}_{\mathcal{S}_\mfp^{=\tau'}(F,t,s)}$ (Lemma \ref{lem:UDimpliesD}).
\end{proof}

\section{Axiomatizing uniform denseness}
\label{sec:axiom}

Let again $K$ be a number field, $\mfp$ a prime of $K$ and $\tau=(e,f)\in\mathbb{N}^{2}$.
As above, if $\mfp$ is a $p$-valuation let $t_\mfp$ be a uniformizer of $v_\mfp$,
and let $q_{\mfp}=|Kv_{\mfp}|$ be the size of the residue field.
When working in a first-order language expanding $\mathcal{L}_{\mathrm{ring}}$ in the following, we freely use multiplicative inversion as though it were represented by a function symbol, as it can be replaced by a either a universal or existential definition for multiplicative inversion. This replacement thus does not affect the quantifier complexity of any formula that is not itself quantifier free.

We will use the following result from \cite{LGP} to axiomatize quantifiers over real and $p$-adic closures.

\begin{Proposition}\label{quantifylocalLemma}
There exists a recursive map 
$\varphi(\mathbf{x})\mapsto\hat{\varphi}_{\mfp}^{\tau}(\mathbf{x})$
from $\mathcal{L}_R$-formulas to $\mathcal{L}_R(K)$-formulas
such that
for every extension $F/K$ and tuple $\mathbf{a}\in F^{m}$
the following holds:
\begin{enumerate}
\item If $(F^\prime,\mathcal{O}_{\mfP'})\models\varphi(\mathbf{a})$ holds for all closures $(F',\mfP')$ of $(F,\mfP)$ 
  with $\mfP\in\mathcal{S}_\mathfrak{p}^{\tau}(F)$,
     then $(F,R_\mfp^{\tau}(F))\models\hat{\varphi}_{\mfp}^{\tau}(\mathbf{a})$.
 \item If $(F,R_\mfp^{\tau}(F))\models\hat{\varphi}_{\mfp}^{\tau}(\mathbf{a})$,
         then $(F',\mathcal{O}_{\mfP'})\models\varphi(\mathbf{a})$ for all
         closures $(F',\mfP')$ of $(F,\mfP)$ with $\mfP\in\mcS_\mfp^{\tau,\mathbb{Z}}(F)$.
\end{enumerate}
\end{Proposition}
 
\begin{proof}
We can copy the proof of \cite[Proposition 8.4]{LGP}, except that
we note that \cite[Lemma 8.3]{LGP} actually only applies \cite[Proposition 8.2]{LGP} and can therefore take an 
$\mathcal{L}_{R}$-formula instead of an $\mathcal{L}_{\rm ring}$-formula,
and that we omit the step of replacing the occurrences of $x\in R$ by the formula $\theta_{R,\mfp}^{\tau'}(x)$,
which makes the application of \cite[Proposition 6.6]{LGP} and therefore the assumption that 
$F$ satisfies $T_{R,\mfp}^{\tau'}$ unnecessary.
\end{proof}

In order to use this result, 
we first deal with the value groups being $\mathbb{Z}$-groups
and then use Proposition \ref{quantifylocalLemma} to axiomatize uniform denseness.

\label{sec:Z}

Recall that an ordered abelian group $G$ is a $\mathbb{Z}$-group
if it is elementarily equivalent to $\mathbb{Z}$;
equivalently, if it is discrete
(i.e.~it has a convex subgroup isomorphic to $\mathbb{Z}$, which we will also denote by $\mathbb{Z}$)
and satisfies 
$G=nG+\{0,\dots,n-1\}$ for all $n\in\mathbb{N}$.
So, given a valuation $v$ on a field $F$ which has a uniformizer, the value group $vF$ is a $\mathbb{Z}$-group if and only if
for each $n\in\mathbb{N}$ and for each $y \in F^\times$ there exists $x\in F^\times$ such that $v(yx^n)\in\{0,\dots,n-1\}\subseteq\mathbb{Z}$,
or, equivalently, there exists $x\in F$ such that $v(yx^n)\in\mathbb{Z}$.

We want to axiomatize the class of fields $F/K$ with $\mathcal{S}_\mfp^\tau(F) = \mathcal{S}_\mfp^{\tau,\mathbb{Z}}(F)$.
In both Lemma \ref{lem:one_value_zero_p_valns}
and Proposition \ref{prop:axioms_Z_groups} we assume that $\mfp$ is a $p$-valuation.

\begin{Lemma}\label{lem:one_value_zero_p_valns}
For any $n\in\mathbb{N}$, there exists a polynomial $\phi_n \in \mathbb{Z}[X_1, \dotsc, X_n]$ such that for any field $F/K$, elements $x_1, \dotsc, x_n \in F$, and $\mfP \in \mathcal{S}_\mfp^\tau(F)$,
we have $v_\mfP(\phi_{n}(x_1, \dotsc, x_n)) = 0$ if and only if $\min(v_\mfP(x_1), \dotsc, v_\mfP(x_n)) = 0$.
\end{Lemma}
\begin{proof}
Let $g \in \mathbb{Z}[X]$ be a monic polynomial whose reduction has no zero in the extension of $K v_\mfp$ of degree $f$; for instance we may take $g$ to be a lift of an irreducible polynomial over $\mathbb{F}_p$ of sufficiently large degree.
Let $\phi_2(X,Y) = Y^{\deg g} g(X/Y) \in \mathbb{Z}[X,Y]$ be the homogenization.
For any $F/K$, $x,y \in F$ and $\mfP \in \mathcal{S}_\mfp^\tau(F)$, we have $v_\mfP(g(x)) = \deg(g) \min(v_\mfP(x), 0)$, and therefore $v_\mfP(\phi_2(x,y)) = \deg(g) \min(v_\mfP(x), v_\mfP(y))$; in particular $\phi_2$ is as desired for $n=2$.
For higher $n$, one uses the polynomial $\phi_2(X_1, \phi_2(X_2, \phi_2(X_3, \dotsc )\dotsb))$.
For $n=1$ we may take $\phi_1=X_1$.
\end{proof}

\begin{Proposition}\label{prop:axioms_Z_groups}
  There exists an $\mathcal{L}_R(K)$-theory $T_{\mathbb{Z},\mfp}^\tau$ such that for every extension $F/K$,
  we have $(F,R_\mfp^\tau(F)) \models T_{\mathbb{Z},\mfp}^\tau$ if and only if $\mathcal{S}_\mfp^\tau(F) = \mathcal{S}_\mfp^{\tau,\mathbb{Z}}(F)$.
\end{Proposition}
\begin{proof}
Let $\nu_{\mfp,n}^\tau$ be the $\mathcal{L}_{R}(K)$-sentence
  \[ \forall y \neq 0 \exists x_0, \dotsc, x_{n-1} \; R^\times( \phi_{n}(y^{e!} t_\mfp^0 x_0^{n}, \dotsc, y^{e!} t_\mfp^{n -1} x_{n - 1}^{n}) ), \]
where $R^\times(x)$ is short for $R(x)\wedge R(x^{-1})$,
and let $T_{\mathbb{Z},\mfp}^\tau=\{\nu_{\mfp,n}^\tau \colon n > 0\}$.

Consider a field $F/K$ with $(F, R_\mfp^\tau(F)) \models T_{\mathbb{Z},\mfp}^\tau$ and let $\mfP \in \mathcal{S}_\mfp^\tau(F)$.
Given $y \in F^\times$ and $n>0$, and letting $n' = e! n$, there exist $i \in \{ 0, \dotsc, n' - 1\}$ and $x \in F$ with $v_\mfP(y^{e!} t_\mfp^i x^{n'}) = 0$ by using $(F, R_\mfp^\tau(F)) \models \nu_{\mfp,n'}$ and Lemma \ref{lem:one_value_zero_p_valns}.
Therefore $v_\mfP(y^{e!} x^{n'}) = e! v_\mfP(y x^n) \in \mathbb{Z}$ and thus $v_\mfP(y x^n) \in \mathbb{Z}$.
Hence $\mfP \in \mathcal{S}_\mfp^{\tau,\mathbb{Z}}(F)$.
This proves $\mathcal{S}_\mfp^\tau(F) = \mathcal{S}_\mfp^{\tau,\mathbb{Z}}(F)$.

For the converse direction, let $F/K$ be a field with $\mathcal{S}_\mfp^\tau(F) = \mathcal{S}_\mfp^{\tau,\mathbb{Z}}(F)$, and fix $n > 1$ and $y \in F^\times$.
We shall find $x_0, \dotsc, x_{n - 1} \in F$ such that for any $\mfP \in \mathcal{S}_\mfp^{\tau}(F)$ we have $v_\mfP(x_i^n) \geq - e! v_\mfP(y) - i v_\mfP(t_\mfp)$  for any $i \in \{0, \dotsc, n - 1\}$, with equality for some $i$.
We see that then $v_\mfP(y^{e!} t_\mfp^i x_i^{n}) \geq 0$ for all $i$, with equality for some $i$, hence $\nu_{\mfp,n}^\tau$ is satisfied by Lemma \ref{lem:one_value_zero_p_valns}.

It remains to prove that such $x_0, \dotsc, x_{n-1}$ can always be found, for which we use Proposition \ref{prop:value_approximation}.
For every $\mfP$, write $e_\mfP = v_\mfP(t_\mfp)$ for the relative ramification index. 
Let $i_\mfP \in \{ 0, \dotsc, n-1\}$ such that $i_\mfP \equiv -\frac{e!}{e_\mfP} v_\mfP(y) \pmod{n}$,
choose $x_\mfP^{(i_\mfP)}$ with $n v_\mfP(x_\mfP^{(i_\mfP)}) = - e! v_\mfP(y) - i_\mfP e_\mfP$.
For $i\in\{0,\dotsc,n-1\}\setminus\{i_{\mfP}\}$, choose $x_\mfP^{(i)}$ with \[-e! v_\mfP(y) - i e_\mfP \leq n v_\mfP(x_\mfP^{(i)}) \leq -e! v_\mfP(y) - (i-n) e_\mfP.\]
For every $\mfP$, the set \[V_\mfP = \{ \mfQ \colon n v_\mfQ(x_\mfP^{(i_\mfP)}) = - v_\mfQ(y^{e!} t_\mfp^{i_\mfP}) \text{ and } -v_\mfQ(y^{e!} t_\mfp^i) \leq n v_\mfQ(x_\mfP^{(i)}) \leq  -v_\mfQ(y^{e!} t_\mfp^{i-n}) \text{ for }i \neq i_\mfP \} \] is an open-closed neighbourhood of $\mfP$ in $\mathcal{S}_\mfp^\tau(F)$.
Since $\mathcal{S}_\mfp^\tau(F)$ is compact by Lemma \ref{lem:SFcompact}, 
the $V_\mfP$ are compact and
there are finitely many $\mfP_1, \dotsc, \mfP_m$ such that the corresponding $V_{\mfP_j}$ cover $\mathcal{S}_\mfp^\tau(F)$, and we may choose open-closed and hence compact subsets $S_j \subseteq V_{\mfP_j}$ that partition $\mathcal{S}_\mfp^\tau(F)$.

By Proposition \ref{prop:value_approximation}, we can find $x_0, \dotsc, x_{n-1}$ such that $v_\mfQ(x_i) = v_\mfQ(x_{\mfP_j}^{(i)})$ for all $\mfQ \in S_j$, all $i\in\{0,\dots,n-1\}$, and all $j\in\{1,\dots,m\}$;
the hypothesis of the proposition is satisfied since for any valuation $v$ which is a proper coarsening of $v_\mfQ$ for some prime $\mfQ \in S_j$ we have $v(x_{\mfP_j}^{(i)}) = -\frac{e!}{n} v(y)$, and the right-hand side is independent of $j$.
These $x_i$ are as desired by construction.
\end{proof}

We return to the setting of an arbitrary prime $\mfp$ of $K$.

\begin{Lemma}\label{lem:defining_sets_primes}
There exists an $\mathcal{L}_R(K)$-formula
$\chi_\mfp^\tau(t,s)$ such that for
every extension $F/K$,
every $\mfP\in\mathcal{S}_\mfp^\tau(F)$,
and every $t,s\in F$, 
the following are equivalent:
\begin{enumerate}
\item $\mfP\in\mathcal{S}_\mfp^{=\tau}(F,t,s)$
\item $(F,\mathcal{O}_\mfP)\models\chi_\mfp^\tau(t,s)$
\item $(F',\mathcal{O}_{\mfP'})\models\chi_\mfp^\tau(t,s)$ for some closure $(F',\mfP')$ of $(F,\mfP)$
\item $(F',\mathcal{O}_{\mfP'})\models\chi_\mfp^\tau(t,s)$ for every closure $(F',\mfP')$ of $(F,\mfP)$
\end{enumerate}
\end{Lemma}

\begin{proof}
If $\mfp$ is an ordering, then
$\mathcal{S}_{\mfp}^{=\tau}(F,t,s)=\mathcal{S}_{\mfp}^{\tau}(F)$, for all $t,s$.
Thus we may take $\chi_{\mfp}^{\tau}(t,s)$ to be $t=t$.
Otherwise, $\mfp$ is a $p$-valuation, and we
let $\chi_\mfp^\tau(t,s)$ be the $\mathcal{L}_R(K)$-formula
\begin{align*}
	R^\times(t^et_\mfp^{-1})\wedge R^\times(s)\wedge \bigwedge_{n | q_\mfp^f-1, n \neq q_\mfp^f-1} R^\times(s^n - 1).
\end{align*}
The equivalence of $(1)$ and $(2)$ follows from Lemma \ref{lem:SFsw}. 
For the equivalence of $(2)$, $(3)$ and $(4)$ note that $\chi_\mfp^\tau$ is equivalent both to an existential and a universal $\mathcal{L}_{R}(K)$-formula, and if $(F',\mfP')$ is a closure of $(F,\mfP)$,
then $\mathcal{O}_{\mfP'}^\times\cap F=\mathcal{O}_\mfP^\times$.
\end{proof}

\begin{Proposition}\label{lem:TStau}
There is a recursive theory $T_{{\rm UD},\mfp}^\tau$ in the language $\mathcal{L}_{R}(K)$ 
such that for every extension $F/K$
the following holds:
\begin{enumerate}
\item
If $F$ satisfies ${\rm UD}_{\mathcal{S}_\mfp^{=\tau}(F,t,s)}$ for every $t,s\in F$, then
$(F,R_\mfp^\tau(F))\models T_{{\rm UD},\mfp}^\tau$.
\item
If $(F,R_\mfp^\tau(F))\models T_{{\rm UD},\mfp}^\tau$, then
$F$ satisfies ${\rm UD}_{\mathcal{S}_\mfp^{=\tau,\mathbb{Z}}(F,t,s)}$ for all $t,s\in F$.
\end{enumerate}
\end{Proposition}

\begin{proof}
For a $n$-tuple $\mathbf{c}$ denote by $g_\mathbf{c}$ the monic polynomial of degree $n$ with coefficients $\mathbf{c}$.
Let $\varphi(\mathbf{c},x,a,t,s)$ be the $\mathcal{L}_R(K)$-formula 
$$
	[\chi_\mfp^\tau(t,s)\wedge(\exists y)(g_\mathbf{c}(y)=0)]\rightarrow R(1-g_\mathbf{c}(x)^2a^{-2}).
$$
Proposition \ref{quantifylocalLemma} gives us an $\mathcal{L}_{R}(K)$-formula $\hat{\varphi}_{\mfp}^\tau(\mathbf{c},x,a,t,s)$.
Let $\psi_n$ be the sentence
$$
 \forall\mathbf{c}\forall a\neq 0\forall t\forall s\exists x\hat{\varphi}_{\mfp}^\tau(\mathbf{c},x,a,t,s)
$$ 
with $\mathbf{c}$ of length $n$.
Then $T_{{\rm UD},\mfp}^\tau=\{\psi_n:n\in\mathbb{N}\}$ is the required $\mathcal{L}_{R}(K)$-theory,
by Lemma \ref{lem:defining_sets_primes}.
\end{proof}

\begin{proof}[Proof of Proposition \ref{prop:axiom}]
Let $T_\mfp^\tau = T_{\mathbb{Z},\mfp}^\tau \cup\bigcup_{\tau'\leq\tau} T_{{\rm UD},\mfp}^{\tau'}$.

If $(F,R_\mfp^\tau(F))\models T_\mfp^\tau$, then 
in particular $(F,R_\mfp^\tau(F))\models  T_{{\rm UD},\mfp}^{\tau'}$ for every $\tau'\leq\tau$,
so $F$ satisfies ${\rm UD}_{\mathcal{S}_\mfp^{=\tau',\mathbb{Z}}(F,t,s)}$ for every $t,s\in F$ 
by Proposition \ref{lem:TStau}.
Since also $(F,R_\mfp^\tau(F))\models T_{\mathbb{Z},\mfp}^\tau$,
we have that $\mathcal{S}_\mfp^{=\tau,\mathbb{Z}}(F,t,s)=\mathcal{S}_\mfp^{=\tau}(F,t,s)$ by
Proposition \ref{prop:axioms_Z_groups}.

Conversely, if ${\rm UD}_{\mathcal{S}_\mfp^{=\tau'}(F,t,s)}$  holds for every $t,s\in F$ and $\tau'\leq\tau$, then in particular ${\rm D}_{\mathcal{S}_\mfp^\tau(F)}$ holds,
hence $\mathcal{S}_\mfp^\tau(F)=\mathcal{S}_\mfp^{\tau,\mathbb{Z}}(F)$ since (in the case of $p$-valuations) 
the value group of a prime in $\mathcal{S}_\mfp^\tau(F)$ is the same as the value group of a corresponding closure by Lemma \ref{lem:denseness_preserves_val_grp_res_fld}, and the latter is a $\mathbb{Z}$-group.
Thus by Proposition \ref{prop:axioms_Z_groups} and Proposition \ref{lem:TStau},
$(F,R_\mfp^\tau(F))\models T_\mfp^\tau$.
\end{proof}

\begin{Remark}
A closer inspection shows that the theory $T_\mfp^\tau$ thus obtained
has quantifier complexity $\forall\exists\forall$.
However, it is not difficult to see that
if $(F_n,R_\mfp^\tau(F_n))_{n\in\mathbb{N}}$
is a chain of models of $T_\mfp^\tau$,
then also $F=\bigcup_{n\in\mathbb{N}}F_n$ 
satisfies
 ${\rm D}_{\mathcal{S}_\mfp^\tau(F)}$
 and so $(F,R_{\mfp}^{\tau}(F))$ is a model of $T_\mfp^\tau$.
 Therefore, by general principles,
$T_\mfp^{\tau}$ is equivalent to an $\forall\exists$-theory.
\end{Remark}

\begin{Remark}
Corollary \ref{cor:Tp} shows that if 
${\rm D}_{\mathcal{S}_\mfp^\tau(F)}$ holds
and $(E,R^{\tau}_{\mfp}(E))\equiv (F,R^{\tau}_{\mfp}(F))$,
as $\mathcal{L}_{R}(K)$-structures,
then also 
${\rm D}_{\mathcal{S}_\mfp^\tau(E)}$
holds.
In Section \ref{sec:counterexample} we show---at least in the case $K=\mathbb{Q}$ and $\mfp=\infty$---that this statement no longer holds true in $\mathcal{L}_{\rm ring}$ instead of $\mathcal{L}_R$,
i.e.~if $F$ is dense in all its real closures and $E\equiv F$,
as $\mathcal{L}_{\mathrm{ring}}$-structures,
it does {\em not} follow that also $E$ is dense in all its real closures.
In other words, the property ${\rm D}_{\mathcal{S}_\mfp^\tau}$ is in general not preserved under elementary equivalence in the language of rings.
\end{Remark}

\begin{Remark}
However, we observe that the property ${\rm D}_{\mathcal{S}_\mfp^\tau}$ does pass down elementary extensions in the language of rings.
Let $F\preceq F^{*}$ be an $\mathcal{L}_{\mathrm{ring}}$-elementary extension
and suppose that ${\rm D}_{\mathcal{S}_{\mfp}^{\tau}(F^{*})}$ holds.
Let $\mfP\in\mathcal{S}_{\mfp}^{\tau}(F)$,
and let $w$ be a non-trivial valuation on $F$ which is a
coarsening of $\mfP$.
There exists an ultrapower $F^{**}$ of $F$
into which we may $\mathcal{L}_{\mathrm{ring}}(F)$-elementarily embed $F^{*}$.
Identifying embeddings with inclusions, we have
$F\preceq F^{*}\preceq F^{**}$.
Since $F^{**}$ is an ultrapower of $F$, we may let $w^{**}$ and $\mfP^{**}$ be the natural extensions of $w$ and $\mfP$ to $F^{**}$, respectively.
By \L os's theorem, we have
$$(F,\mathcal{O}_{w},\mathcal{O}_{\mfP})\preceq(F^{**},\mathcal{O}_{w^{**}},\mathcal{O}_{\mfP^{**}}),$$
where these are viewed as structures in a language $\mathcal{L}_{R,S}(K)$ which is the expansion of $\mathcal{L}_{\mathrm{ring}}(K)$ by two unary predicates $R$ and $S$.
In particular: $\mfP^{**}$ is also a prime of $F^{**}$, lying above $\mfp$, and of the same exact type $\tau'\leq\tau$ as $\mfP$;
and $w^{**}$ is a non-trivial valuation on $F^{**}$ coarsening $\mfP^{**}$.
Next, we let $w^{*}$ and $\mfP^{*}$ denote the restrictions of $w^{**}$ and $\mfP^{**}$ to $F^{*}$,
respectively.
Then
\begin{align}\label{eq.ex.cl}
\tag{$\dagger$}
	(F,\mathcal{O}_{w},\mathcal{O}_{\mfP})\text{ is existentially closed in }(F^{*},\mathcal{O}_{w^{*}},\mathcal{O}_{\mfP^{*}})
\end{align}
as $\mathcal{L}_{R,S}(K)$-structures.
Again 
$\mfP^{*}$ is of the same exact type as $\mfP$,
and $w^{*}$ is a coarsening of $\mfP^{*}$.
In particular $\mfP^{*}\in\mathcal{S}_{\mfp}^{\tau}(F^{*})$.

We break into two cases.
First, suppose that $\mfp$ is a $p$-valuation.
Then $v_{\mfP}$
induces a $p$-valuation $\bar{v}_{\mfP}$
on $Fw$,
and $v_{\mfP^{*}}$
induces a $p$-valuation $\bar{v}_{\mfP^{*}}$
on $F^{*}w^{*}$.
In fact $\bar{v}_{\mfP}$ is the restriction of $\bar{v}_{\mfP^{*}}$ to $Fw$.
By assumption, $F^{*}$ is dense in a closure of $(F^{*},\mfP^{*})$;
and thus $v_{\mfP^{*}}F^{*}$ is a $\mathbb{Z}$-group and
$(F^{*}w^{*},\bar{v}_{\mfP^{*}})$ is henselian,
the latter by $(1)\Rightarrow(2)$ of Lemma \ref{lem:Ershov}.
From (\ref{eq.ex.cl}),
it follows that
$v_{\mfP}F$ is existentially closed in $v_{\mfP^{*}}F^{*}$,
thus $v_{\mfP}F$ is also a $\mathbb{Z}$-group.
Also from 
(\ref{eq.ex.cl}),
it follows that $Fw$ is existentially closed in $F^{*}w^{*}$.
In particular, $Fw$ is relatively algebraically closed in $F^{*}w^{*}$,
and so $\bar{v}_{\mfP}$ is henselian.
By $(2)\Rightarrow(1)$ of Lemma \ref{lem:Ershov}, and since $w$ was an arbitrary non-trivial valuation coarsening $v_{\mfP}$,
$F$ is dense in a henselization of $(F,\mfP)$, which is a $p$-adic closure.
Note that, by ignoring $w$ altogether, this argument shows that
the property
$\mathcal{S}_{\mfp}^{\tau}=\mathcal{S}_{\mfp}^{\tau,\mathbb{Z}}$
passes down elementary extensions in the language of rings.

If $\mfp$ is an ordering,
on the other hand,
then $\mfP$ induces an ordering $\leq_{w}$ on $Fw$,
and $\mfP^{*}$
induces an ordering $\leq_{w^{*}}$
on $F^{*}w^{*}$.
Again, $\leq_{w}$ is the restriction of $\leq_{w^{*}}$ to $Fw$.
Again, by assumption, $F^{*}$ is dense in a closure of $(F^{*},\mfP^{*})$;
and thus $w^{*}F^{*}$ is divisible and $(F^{*}w^{*},\leq_{w^{*}})$ is real closed,
by Proposition \ref{prp:dense.in.real.closure}.
From (\ref{eq.ex.cl}), it follows that
$wF$ is existentially closed in $w^{*}F^{*}$,
and that $(Fw,\leq_{w})$ is existentially closed in $(F^{*}w^{*},\leq_{w^{*}})$.
Thus $wF$ is divisible and $(Fw,\leq_{w})$ is real closed.
By Proposition \ref{prp:dense.in.real.closure}, and since $w$ was an arbitrary non-trivial $\leq_{\mfP}$-convex valuation, $F$ is dense in a real closure of $(F,\mfP)$.
Thus, in either case, ${\rm D}_{\mathcal{S}_{\mfp}^{\tau}(F)}$ holds.
\end{Remark}

\begin{Remark}
It is not hard to see that
denseness of $F$ in the closure
of one specific prime $\mfP$ of $F$ is an elementary property
in $\mathcal{L}_R$ where $R$ interprets $\mathcal{O}_\mfP$:
If $\mfP$ is a valuation then (3) of Lemma \ref{lem:Dense_in_henselization_polynomials} gives an elementary characterization of denseness of $(K,v)$ in its henselization, and one axiomatizes that a henselization is a closure by demanding that the value group be a $\mathbb{Z}$-group.
If $\mfP$ is an ordering then a similar characterization is
given in \cite[Satz 13]{Hau}.
\end{Remark}

\section{Algebraic fields and applications}
\label{sec:algebraic}

Recall that the Pythagoras number of a field $F$ is the smallest $n$ such that
every sum of squares in $F$ is a sum of at most $n$ squares. 
We denote the Pythagoras number of $F$ by $\pi_\infty(F)$. 
Let $\Phi_{\infty,n}$ be the $\mathcal{L}_{\rm ring}$-theory stating that every sum of $n+1$ squares is a sum of $n$ squares,
and let $\varphi_{\infty,n}(x)$ be an $\mathcal{L}_{\rm ring}$-formula defining the sums of $n$ squares.
As the sums of squares of $F$ are precisely the elements of $R_\infty(F)$,
the following is obvious:

\begin{Lemma}\label{lem:phi_infty}
For every field $F$, and every $n$, we have $\varphi_{\infty,n}(F)\subseteq  R_\infty(F)$.
Moreover, the following are equivalent:
\begin{enumerate}
\item $\varphi_{\infty,n}(F)=R_\infty(F)$ 
\item $\pi_\infty(F)\leq n$.
\item $F\models\Phi_{\infty,n}$
\end{enumerate}
\end{Lemma}

For a $p$-valuation $\mfp$ on a number field $K$, and $\tau\in\mathbb{N}^{2}$,
\cite{ADF2} defines a $p$-adic analogue $\pi_\mfp^\tau(F)$ of the Pythagoras number,
as follows.
For each $F/K$,
we first define an increasing chain $(R_{\mfp,n}^{\tau}(F))_{n\in\mathbb{N}}$ of
subsets of $F$,
as in \cite[Section 2.2]{ADF2},
such that
\begin{align*}
    R_{\mfp}^{\tau}(F)&=\bigcup_{n\in\mathbb{N}}R_{\mfp,n}^{\tau}(F).
\end{align*}
Then the $(\mfp,\tau)$-Pythagoras number of $F$ is
\begin{align*}
	\pi_{\mfp}^{\tau}(F)&:=\inf\big\{n\in\mathbb{N}\;\big|\;R_{\mfp}^{\tau}(F)=R_{\mfp,n}^{\tau}(F)\big\}\in\mathbb{N}\cup\{\infty\}.
\end{align*}
In the language of \cite[Section 3]{ADF2}, for each $n$, the map
$F\longmapsto R_{\mfp,n}^{\tau}(F)$
is a $1$-dimensional diophantine family,
by \cite[Example 3.8]{ADF2}.
That is, for each $n$, there is an existential $\mathcal{L}_{\mathrm{ring}}(K)$-formula $\varphi_{\mfp,n}^{\tau}(x)$ such that
$\varphi_{\mfp,n}^{\tau}(F)=R_{\mfp,n}^{\tau}(F)$, for all $F/K$.
Let $\Phi_{\mfp,n}^{\tau}$ be the universal-existential $\mathcal{L}_{\mathrm{ring}}(K)$-theory consisting of the sentences
\begin{align*}
	\forall x\;[\phi_{\mfp,m}^{\tau}(x)\rightarrow\phi_{\mfp,n}^{\tau}(x)],
\end{align*}
for all $m\geq n$.
Just as for orderings, the following is now obvious.

\begin{Lemma}\label{lem:phi}
If $F/K$ is an extension, then $\varphi_{\mfp,n}^\tau(F)\subseteq  R_\mfp^\tau(F)$,
and the following are equivalent:
\begin{enumerate}
\item $\varphi_{\mfp,n}^\tau(F)=R_\mfp^\tau(F)$ 
\item $\pi_\mfp^\tau(F)\leq n$.
\item $F\models\Phi_{\mfp,n}^\tau$
\end{enumerate}
\end{Lemma}

Let $\mfp$ now be an arbitrary prime of $K$.
For a class of fields $\mathcal{F}$ we let $\pi_\mfp^\tau(\mathcal{F})=\sup_{F\in\mathcal{F}}\pi_\mfp^\tau(F)$.

\begin{Lemma}\label{lem:bounded_pi_closure}
For a class of fields $\mathcal{F}$ let
$\mathcal{F}^*$ denote its closure
under elementary equivalence, direct limits and ultraproducts.
Then $\pi_\mfp^\tau(\mathcal{F}^*)=\pi_\mfp^\tau(\mathcal{F})$.
\end{Lemma}

\begin{proof}
This is clear, 
since $\Phi_{\mfp,n}^\tau$ is a $\forall\exists$-theory.
\end{proof}

Let
$T_{\mfp,n}^\tau$ denote the union of $\Phi_{\mfp,n}^\tau$
and $T_\mfp^\tau$, with all occurrences of $R$ replaced with $\varphi_{\mfp,n}^\tau(x)$.

\begin{Proposition}\label{prop:Tpn}
$F\models T_{\mfp,n}^\tau$ if and only if
$\pi_\mfp^\tau(F)\leq n$ and ${\rm D}_{\mathcal{S}_\mfp^\tau(F)}$ holds.
\end{Proposition}

\begin{proof}
This follows from Corollary \ref{cor:Tp} together with
Lemma \ref{lem:phi}.
\end{proof}

\begin{Corollary}
If $\pi_\mfp^\tau(F)=n<\infty$
and $F$ is dense in all closures at elements of $\mathcal{S}_\mfp^\tau(F)$,
then every $F'$ which is $\mathcal{L}_{\mathrm{ring}}(K)$-elementarily equivalent to $F$ is dense in all closures at elements of $\mathcal{S}_\mfp^\tau(F')$.
\end{Corollary}

We denote by $\mathcal{A}$ the class of algebraic fields of characteristic zero,
i.e.~fields $F\subseteq\mathbb{Q}^{\rm alg}$.
We denote by ${\rm Th}(\mathcal{A})=\bigcap_{F\in\mathcal{A}}{\rm Th}_{\mathcal{L}_{\rm ring}}(F)$ 
the $\mathcal{L}_{\mathrm{ring}}$-theory of $\mathcal{A}$.

\begin{Proposition}
For each prime $p$ of $\mathbb{Q}$ and pair $\tau$ of positive integers,
there exists $n_p^\tau\in\mathbb{N}$ such that $\pi_p^\tau(F)\leq n_p^\tau$
for each field $F\in\mathcal{A}$.
\end{Proposition}

\begin{proof}
For $F$ in the class $\mathcal{K}$ of number fields,
the statement is \cite[Thm 1.2]{ADF2}.
By Lemma \ref{lem:bounded_pi_closure} it then holds for the 
class $\mathcal{K}^\ast$, which contains $\mathcal{A}$.
\end{proof}

\begin{Corollary}
\label{cor:denseness}
Every $F'\models{\rm Th}(\mathcal{A})$ is dense in all its real closures and all its $p$-adic closures {\rm(}of arbitrary $p$-rank{\rm)}.
\end{Corollary}

\begin{proof}
Let $K=\mathbb{Q}$ and let $T$ be 
the union of $T_{p,n_p^\tau}^\tau$
for all primes $p$ of $\mathbb{Q}$
and all $\tau\in\mathbb{N}^2$.
Then every $F\in\mathcal{A}$ is a model of $T$:
For every $p$ and $\tau$, 
every  $\mfP\in\mathcal{S}_p^\tau(F)$ is either an archimedean ordering or
a valuation with value group $\mathbb{Z}$, 
so $F$ is dense in a closure of $(F,\mfP)$.
By Proposition \ref{prop:Tpn}, $F\models T_{p,n_p^\tau}^\tau$.

Therefore $T\subseteq{\rm Th}(\mathcal{A})$ 
and so also every $F'\models{\rm Th}(\mathcal{A})$ is a model
of $T_{p,n_p^\tau}^\tau$ for every $p\in\mathcal{S}(\mathbb{Q})$ and $\tau\in\mathbb{N}^{2}$.
Let $\mfP$ be a prime of $F'$.
If $p$ is the restriction of $\mfP$ to $\mathbb{Q}$
and $\tau$ denotes the relative type of $\mfP$ over $p$,
then $\mfP\in\mathcal{S}_p^\tau(F')$.
As $F'\models T_{p,n_p^\tau}^\tau$,
Proposition \ref{prop:Tpn} gives that
${\rm D}_{\mathcal{S}_\mfp^\tau(F')}$ holds.
In particular, $F'$ is dense in every closure at $\mfP$.
\end{proof}

\begin{Remark}
After this work was finished we realized that
for $p$-adic closures
this corollary can in fact also be deduced
from the literature in a different way:
For each $F\in\mathcal{A}$, the ring $R_p^\tau(F)$ satisfies 
Darni\`{e}re's {\em Block Approximation Property (BAP)}
\cite[II.2 p.~29]{Darniere}
and is dense in all of its henselizations,
i.e.~$F$ is dense in the quotient fields of henselizations of 
the localizations of $R_p^\tau(F)$ at maximal ideals,
which in this case are precisely the $p$-adic closures of $F$.
As this class of rings is elementary by \cite[II.5 Prop.~5]{Darniere},
and $\varphi_{p,n_p^\tau}^\tau$ defines $R_p^\tau(F')$
in every model $F'$ of ${\rm Th}(\mathcal{A})$,
$R_p^\tau(F')$ will be dense in all its henselizations.
However, also in this case the quotient fields of these henselizations are precisely the $p$-adic closures of $F'$:
Indeed, for a henselian $p$-valued field $E$ with value group $\Gamma$, 
the $\ell$-cohomological dimension is $1 + \dim_{\mathbb{F}_\ell} \Gamma/\ell\Gamma$, in particular this is bigger than 2 for some prime number $\ell$ if $\Gamma$ is not a $\mathbb{Z}$-group,
i.e.~if $E$ is not p-adically closed,
and by \cite[p.~129 Theorem]{Chatzidakis}, 
every model $F'$ of ${\rm Th}(\mathcal{A})$
has (virtual) cohomological dimension at most 2,
and therefore so has every algebraic extension $E$ of $F'$.
\end{Remark}

\begin{Remark}
By Proposition \ref{prop:D_implies_UD} and
Remark \ref{rem:examples_compact},
${\rm D}_{\mathcal{S}_{\mfp}^{\tau}(F)}$
implies the uniform version ${\rm UD}_{\mathcal{S}_{\mfp}^{=\tau}(F,t,s)}$,
for each $t,s$.
We note that the proof of this fact generalizes easily to
allow the prime $\mfp$ to vary within a given finite set
$S$ of primes of $K$.
More precisely, if we denote
$\mathcal{S}_{S}^{\tau}(F)=\bigcup_{\mfp\in S}\mathcal{S}_{\mfp}^{\tau}(F)$
and $\mathcal{S}_{S}^{=\tau}(F,t,s)=\bigcup_{\mfp\in S}\mathcal{S}_{\mfp}^{=\tau}(F,t,s)$;
then
${\rm D}_{\mathcal{S}_{S}^{\tau}(F)}$
implies
${\rm UD}_{\mathcal{S}_{S}^{=\tau}(F,t,s)}$,
for each $t,s$.

In particular, every model $F$ of the theory of algebraic fields satisfies
${\rm UD}_{\mathcal{S}_{S}^{=\tau}(F,t,s)}$, for every $t,s$, and every finite set $S$.
However, for models of the common theory of number fields one can say more:
if $F^{*}$ is a model of the common $\mathcal{L}_{\mathrm{ring}}(K)$-theory of number fields extending $K$,
and $S$ is a finite set of primes of $K$,
then $F^*$ satisfies ${\rm UD}_{S'}$ where
$S'=\bigcup_{\tau\in\mathbb{N}^2}\mathcal{S}_S^\tau(F^*)$.
\end{Remark}

\begin{Remark}
  As we already observed in the introduction, our denseness result has long been known for PRC fields and P$p$C fields, which are models of the theory of algebraic fields.
  In a sense the archetypal result of this kind is the Frey--Prestel theorem \cite[Proposition 11.5.3]{FJ} -- every PAC field is dense with respect to any non-trivial valuation in its algebraic closure --, from which one in particular deduces that henselizations of PAC fields are separably closed.

  There are also other fields in the literature which are dense in all their real and $p$-adic closures in a non-trivial way, for instance coming from other local--global principles which are not subsumed by Pop's notion of PCC fields: for example, the regularly closed fields of \cite{HeinemannPrestel} are dense in their real and $p$-adic closures, but they are not always models of the theory of algebraic fields.
\end{Remark}

We now draw a few simple consequences from Corollary \ref{cor:denseness}.

\begin{Corollary}\label{cor:application_1}
Let $F$ be a model of the theory of algebraic fields, 
and let $v$ be a non-trivial valuation on $F$.
If $Fv$ is formally real, respectively formally $p$-adic,
then $Fv$ is real closed, resp.~$p$-adically closed.
\end{Corollary}

\begin{proof}
If $Fv$ has a prime $\mfQ$, this is induced from a prime $\mfP$ of $F$ of which $v$ is a coarsening: 
In the case of $p$-valuations, this is simply composition of places,
and in the case of orderings it follows from the Baer--Krull theorem.
Now $F$ is dense in every closure with respect to $\mfP$ by Corollary \ref{cor:denseness}.
If $\mfQ$ is an ordering, Proposition \ref{prp:dense.in.real.closure} implies that $(Fv, \leq_\mfQ)$ is real closed.
If $\mfQ$ is a $p$-valuation, Lemma \ref{lem:Ershov} implies that $(Fv, v_\mfQ)$ is henselian, and Lemma \ref{lem:denseness_preserves_val_grp_res_fld} implies that the value group $v_\mfP F$, and therefore also its convex subgroup $v_\mfQ(Fv)$, is a $\mathbb{Z}$-group. Hence $(Fv, v_\mfQ)$ is $p$-adically closed.
\end{proof}

\begin{Remark}\label{rem:dense_henselization}
Algebraic fields are not just dense in all $p$-adic closures, which are the henselizations with respect to $p$-valuations,
but in fact in {\em all} their henselizations. 
We note that this stronger property does {\em not} transfer to models of the theory of algebraic fields:
In fact, already any proper elementary extension $\mathbb{Q}^*$ of $\mathbb{Q}$ has many valuations
for which this fails, one example being the following:
Like in any
ordered
field,
as in Section \ref{sec:independence},
the convex closure $\mathcal{O}$ of $\mathbb{Z}\subseteq\mathbb{Q}^*$ is the valuation ring of a valuation $v=v_{\leq}$ with residue field $F:=\mathbb{Q}^*v\subseteq\mathbb{R}$.
Since $\mathbb{Q}^*$ is a proper elementary extension of $\mathbb{Q}$, the valuation $v$ is non-trivial, and hence $F$ is real closed by Corollary \ref{cor:application_1}.
Let $\bar{w}$ be any non-trivial valuation on $F$ and note that $\bar{w}$ is necessarily non-henselian.
Then $\mathbb{Q}^*$ is not dense in a henselization of the composition $w=\bar{w}\circ v$, as $w$
induces the non-henselian valuation $\bar{w}$ on the residue field of its non-trivial coarsening $v$,
see Lemma \ref{lem:Ershov}.
\end{Remark}

\begin{Corollary}
Let $F$ be an algebraic field that admits a prime $\mfP$ such that $(F,\mfP)$ is neither real closed nor $p$-adically closed.
If $F\prec F^*$ is an elementary extension and $v$ a non-trivial valuation on $F^*$, 
then there is no embedding of $F^*v$ into $F^*$.
\end{Corollary}

\begin{proof}
Suppose that $F^*v\subseteq F^*$.
There is a prime $\mfP^{*}$ of $F^{*}$ of the same type as $\mfP$, i.e.\ either both $\mfP,\mfP^{*}$ are orderings, or both $\mfP,\mfP^{*}$ are $p$-valuations of the same exact type $\tau$
(see \cite[Theorem 6.4]{PR} for the case of $p$-valuations).
Restricting $\mfP^*$ to $F^*v$ gives a prime $\mfp$ of $F^*v$,
so $F^*v$ is real or $p$-adically closed by Corollary \ref{cor:application_1}.
In particular, $F^*$ contains a real or $p$-adically closed field,
which implies that also $F=F^*\cap\mathbb{Q}^{\rm alg}$ is real or $p$-adically closed,
contradicting the assumption.
\end{proof}

In the language of \cite{AnscombeFehm}, this means that an algebraic field $F$ as in the corollary does not have {\em embedded residue}
when viewed as a $\mathbb{Z}$-field.
By \cite[Corollary 5.3]{AnscombeFehm}, this immediately implies the following,
where we write $\mathbb{Q}_{p,{\rm alg}}=\mathbb{Q}_p\cap\mathbb{Q}^{\rm alg}$, $\mathbb{R}_{\rm alg}=\mathbb{R}\cap\mathbb{Q}^{\rm alg}$:

\begin{Corollary}
Let $F$ be a field with $F\subsetneqq\mathbb{R}_{{\rm alg}}$
or $[F\mathbb{Q}_{p,{\rm alg}}:\mathbb{Q}_{p,{\rm alg}}]<\infty$ and $\mathbb{Q}_{p,{\rm alg}}\not\subseteq F$ for some $p$.
Then there exists an existential formula $\phi(x)$ in the language of rings which defines in $F((t))$ the valuation ring $F[[t]]$.
\end{Corollary}

We note that the assumption 
 $F\subsetneqq\mathbb{R}_{{\rm alg}}$
or $[F\mathbb{Q}_{p,{\rm alg}}:\mathbb{Q}_{p,{\rm alg}}]<\infty$ and $\mathbb{Q}_{p,{\rm alg}}\not\subseteq F$
cannot be dropped: For example,
the statement does not hold for $F=\mathbb{R}_{\rm alg}$
or $F=\mathbb{Q}_{p,{\rm alg}}$, see \cite[Corollary 6.6]{AnscombeFehm}.

\section{Denseness in real closures in the language of rings}
\label{sec:counterexample}

Let $F$ be a field of characteristic zero.
By Proposition \ref{prop:Tpn}, if $\pi_\infty(F)<\infty$, then denseness of $F$ in its real closures
transfers to every $F'$ which is elementarily equivalent to $F$ in the language of rings.
In this section we adapt a construction of Prestel from \cite{Prestel1978}
to show that this can fail in the case $\pi_\infty(F)=\infty$.
In fact we will see that denseness in real closures transfers in general not to all ultrapowers $F^*$ of $F$.

\begin{Proposition}
There exists a field $K_\infty\subseteq\mathbb{R}$ 
which has a unique ordering,
and for each $i\in\mathbb{N}$ has a valuation $u_{i}$ with residue field $K_{\infty}u_{i}$ which is
hilbertian of finite level\footnote{Recall that the {\em level} $s(F)$
of a field $F$ is
the least number of squares that sum to $-1$.} $s(K_\infty u_i)\geq 4^i$.
\end{Proposition}

\begin{proof}
We construct a chain of countable fields
$K_0 \subseteq K_1 \subseteq K_2 \subseteq \dots$
and for each $i$ extensions $H_i/K_i$ and $k_{i+1}/K_i$
and elements $d_i\in K_i$ such that
\begin{enumerate}
\item $d_i$ is a sum of $2\cdot 4^i+1$ squares in $K_i$
\item $K_i$ embeds into $\mathbb{R}$ and has a unique ordering
\item $H_i$ carries a henselian valuation $v_i$ trivial on $K_{i-1}$ with residue field $H_iv_i=K_iv_i=k_i$
\item $k_i$ is hilbertian with $4^i\leq s(k_i)\leq 3\cdot 4^i$
\item $k_i$ carries a valuation $w_i$ trivial on $K_{i-1}$ with residue field contained in $H_{i-1}$ 
\end{enumerate}
We begin with $K_{-1} = \mathbb{R}_{\rm alg}$, $d_{-1} = 1$.
Given $K_{i-1}$, $H_{i-1}$, $k_{i-1}$ and $d_{i-1}$ satisfying (1)-(5), let $F_i = K_{i-1}(x_{i,1}, \dotsc, x_{i,4^i})$ for some new indeterminates $x_{i,j}$, $L_i = F_i(y_i)$ for a new indeterminate $y_i$, $e_i = d_{i-1} + \sum_{j=1}^{4^i} x_{i,j}^2 \in F_i$ and $d_i = e_i + y_i^2 \in L_i$.
Note that
$d_i$ is a sum of $2\cdot 4^{i-1}+1+4^i+1\leq 2\cdot 4^i+1$ squares.
The field $L_i$ carries a valuation $v_i$, coming from the prime ideal $(y_i^2 + e_i)$ of $F_i[y_i]$ (prime since $e_i$ is totally positive), with residue field $k_i := F_i(\sqrt{-e_i})$; this $k_i$ is hilbertian since it is a finitely generated transcendental extension of $K_{i-1}$,
see~\cite[Chapter 13]{FJ}.
Now $e_i$ cannot be written as a sum of fewer than $4^i$ squares in $F_i$ by a theorem of Cassels and Pfister
\cite[Corollary IX.2.3]{Lam},
which implies that $4^i\leq s(k_i)<\infty$, cf~\cite[p.~421 Exercise 5]{Lam}.
As $e_i$ is a sum of fewer than $3\cdot 4^i$ squares in $F_i$,
$s(k_i)\leq 3\cdot 4^i$.

We construct $H_i$ as a henselization of $L_i$ with respect to $v_i$, and $K_i$ as the intersection of $H_i$ with a real closure $R_i$ of $L_i$ with respect to an archimedean ordering.
(Evidently $L_i$ has an archimedean ordering since it can be embedded into $\mathbb{R}$, since $K_{i-1}$ can by assumption.)
By \cite[Lemma 1.3]{Prestel1978}, $K_i$ is then uniquely ordered.
As $L_i\subseteq K_i\subseteq H_i$ 
we get that $H_iv_i=K_iv_i=L_iv_i=k_i$.

We now construct a valuation on $k_i$ (for $i>1$), trivial on $K_{i-1}$, with residue field contained in $H_{i-1}$.
By Hensel's Lemma, $s(k_{i-1})\leq 3\cdot 4^{i-1}$
implies that also $s(H_{i-1})\leq 3\cdot 4^{i-1}$.
Therefore we can find $z_{i,1}, \dotsc, z_{i,4^i} \in H_{i-1}$ such that $y_{i-1}^2 + z_{i,1}^2 + \dotsb + z_{i,4^i}^2 = 0$.
On $F_i = K_{i-1}(x_{i,1}, \dotsc, x_{i,4^i})$, take a $H_{i-1}$-valued place\footnote{We use the same symbol for a valuation and the associated place.} $w$, trivial on $K_{i-1}$, under which each $x_{i,j}$ has residue $z_{i,j}$. 
Then the residue of $e_i = e_{i-1} + y_{i-1}^2 + x_{i,1}^2 + \dotsb + x_{i,4^i}^2$ is $e_{i-1}$, whose negative has a square root in $H_{i-1}$ by Hensel's Lemma; therefore we can extend $w$ to an $H_{i-1}$-valued place $w_i$ on $k_i$.

Let $K_\infty = \bigcup_{i\in\mathbb{N}} K_i$. This field is naturally embedded into $\mathbb{R}$ by using the embeddings of all the $K_i$, and it is uniquely ordered since all the $K_i$ are, by (2).
It remains to show that
for each $i_0$, there is a valuation on $K_\infty$ with residue field $k_{i_0}$.
For each $i > i_0$, we have a place $w_i\circ v_i:H_i \dasharrow k_i \dasharrow H_{i-1}$ trivial on $K_{i-1}$, by (3) and (5); 
by composition, we obtain a place $H_i \dasharrow k_{i_0}$ and therefore a place $K_i \dasharrow k_{i_0}$, with residue field exactly $k_{i_0}$.
The corresponding valuations on the $K_i$ are all compatible, since $w_i\circ v_i$ is trivial on $K_{i-1}$; therefore we can take an inductive limit of valuation rings to obtain a valuation $u_{i_0}$ on $K_\infty$ with residue field $k_{i_0}$.
\end{proof}

\begin{Lemma}\label{lem:Dsumsofsquares}
${\rm D}_{\mathcal{S}_\infty(F)}$ implies that
for every $g \in F[X]$ of odd degree and every $\varepsilon \in F^\times$, there exists $x \in F$ with $\varepsilon^2-g(x)^2$
a sum of squares in $F$.
\end{Lemma}

\begin{proof}
By Proposition \ref{prop:DimpliesUD}, ${\rm D}_{\mathcal{S}_\infty(F)}$ implies
${\rm UD}_{\mathcal{S}_\infty(F)}$.
Since $g$ has a zero in every real closure of $F$,
${\rm UD}_{\mathcal{S}_\infty(F)}$
gives an $x\in F$ with $1-g(x)^2\varepsilon^{-2}\in R_\infty(F)$.
As $R_\infty(F)$ consists of the sums of squares of $F$,
we get that $\varepsilon^2-g(x)^2$ is a sum of squares.
\end{proof}

\begin{Lemma}\label{lemmaFindingPolyAndEpsilon}
  Let $s \geq 2$ and
  let $v$ be a valuation on $F$ whose residue field $Fv$ has level $s(Fv)\geq s$.
  Let $g\in\mathcal{O}_v[X]$ for which the reduction
  $\overline{g}\in Fv[X]$ has no root in $Fv$.
  Then 
  for no $x\in F$ and $\varepsilon \in F^\times$ with $v(\varepsilon)>0$
  is $\varepsilon^2 - g(x)^2$ a sum of $s-1$ squares in $F$.
\end{Lemma}

\begin{proof}
  Suppose for contradiction that $\varepsilon^2 = g(x)^2 + y_1^2 + \dotsb + y_{s-1}^2$, and look for terms of lowest valuation on the right-hand side.
  Since $v(\varepsilon^2)>0\geq v(g(x)^2)$ by the assumption on $\overline{g}$, and no non-trivial sum of $s$ squares in $Fv$ is zero, we obtain a contradiction.
\end{proof}

\begin{Corollary}
The field $K_\infty$ is dense in all its real closures
but has an elementary extension which is not.
\end{Corollary}

\begin{proof}
As all orderings of $K_\infty$ are archimedean,
it is dense in its real closures.

Let $\varphi_n(a,\varepsilon)$ be a formula expressing
that $\varepsilon\neq 0$ and with $g_a=X^3-a$, there is no $x$ such that $\varepsilon^2-g_a(x)^2$
is a sum of $4^n-1$ squares.
Then $\{\varphi_1,\dots,\varphi_n\}$ is satisfiable in $K_\infty$:
As $K_\infty u_n$ is hilbertian there exists $a\in\mathcal{O}_{u_n}$
such that $X^3-\bar{a}$ has no zero in $K_\infty u_n$.
By Lemma \ref{lemmaFindingPolyAndEpsilon}, for every $\varepsilon\in K_\infty^\times$ with $u_n(\varepsilon)>0$, $K_\infty\models\varphi_n(a,\varepsilon)$.

Therefore, if $K_\infty\prec K^*$ is an $\aleph_0$-saturated extension,
$\{\varphi_1,\varphi_2,\dots\}$ is realized in $K^*$,
i.e.~there exists $0\neq\varepsilon\in K^*$ and $a\in K^*$
such that $\varepsilon^2-f_a(x)^2$ is not a sum of squares for any $x$.
By Lemma \ref{lem:Dsumsofsquares} this shows that $K^*$ is not dense in all its real closures.
\end{proof}

\section*{Acknowledgements}

Part of this work was done while all three authors were guests of the Institut Henri Poincar\'{e}. 
The authors would like to thank the IHP for funding and hospitality,
and the organizers of the trimester `Model theory, combinatorics and valued fields’ for the invitation.

S.~A.~was supported by the Leverhulme Trust under grant RPG-2017-179.
P.~D.~was supported by KU Leuven IF C14/17/083.
A.~F.~was funded by the Deutsche Forschungsgemeinschaft (DFG) - 404427454.

\end{document}